\documentclass[11pt]{amsart}
\usepackage{lmodern}
\usepackage[T1]{fontenc}
\usepackage{microtype}
\usepackage{amssymb}
\usepackage{amsthm}
\usepackage{amscd}
\usepackage{hyperref}
\usepackage{cite}
\usepackage{color}

\def\equalinlaw{\mathop{=_{\mathcal{D}}}\nolimits}			
\def\age{\mathop{\text{age}}\nolimits}					

\def\Nb{\mathop{\mathbb{N}_{}}\nolimits}					
\def\ar{\mathop{\mathrm{ar}}\nolimits}	
\def\tp{\mathop{\mathrm{tp}}\nolimits}	
\def\odap{\mathord{<}\omega}

\newtheorem{theorem}{Theorem}[section]
\newtheorem{lemma}[theorem]{Lemma}

\newtheorem{cor}[theorem]{Corollary}

\theoremstyle{definition}
\newtheorem{definition}[theorem]{Definition}

\newtheorem{example}[theorem]{Example}

\def \rng{\operatorname{rng}}

\makeatletter

\def\dotminussym#1#2{%
  \setbox0=\hbox{$\m@th#1-$}%
  \kern.5\wd0%
  \hbox to 0pt{\hss\hbox{$\m@th#1-$}\hss}%
  \raise.6\ht0\hbox to 0pt{\hss$\m@th#1.$\hss}%
  \kern.5\wd0}

\mathchardef\mhyphen="2D

\allowdisplaybreaks[2]

\begin{document}

\title{Relative exchangeability with equivalence relations}
\author{Harry Crane and Henry Towsner}
\address {Department of Statistics \& Biostatistics, Rutgers University, 110 Frelinghuysen Avenue, Piscataway, NJ 08854, USA}
\email{hcrane@stat.rutgers.edu}
\urladdr{\url{http://stat.rutgers.edu/home/hcrane}}
\address {Department of Mathematics, University of Pennsylvania, 209 South 33rd Street, Philadelphia, PA 19104-6395, USA}
\email{htowsner@math.upenn.edu}
\urladdr{\url{http://www.math.upenn.edu/~htowsner}}
\thanks{H.\ Crane is partially supported by NSF grant DMS-1308899.}


\date{\today}

\begin{abstract}
We describe an Aldous--Hoover-type characterization of random relational structures that are exchangeable relative to a fixed structure which may have various equivalence relations.  Our main theorem gives the common generalization of the results on relative exchangeability due to Ackerman \cite{Ackerman2015} and Crane and Towsner \cite{CraneTowsner2015} and hierarchical exchangeability results due to Austin and Panchenko \cite{AustinPanchenko2014}.
\end{abstract}

\maketitle

\section{Introduction}\label{section:introduction}

The Aldous--Hoover theorem \cite{Aldous1981,Hoover1979} gives a canonical form for random countable structures derived from exchangeable processes.  
Subsequent work on relatively exchangeable structures \cite{Ackerman2015,CraneTowsner2015} and hierarchically exchangeable arrays \cite{AustinPanchenko2014} refine the Aldous--Hoover theory in several directions.
Here we consider a further extension, not covered by the main theorems of \cite{Ackerman2015,AustinPanchenko2014,CraneTowsner2015}, to the case where symmetries are determined by a structure with definable equivalence relations.
We simplify the exposition specializing to symmetric structures until Section \ref{section:asymmetric}, at which point we discuss our handling of asymmetries and then state our general representation theorem.

We first recall some definitions.  By a (finite, relational) \emph{signature} we mean a finite set of symbols $\mathcal{L}=\{R_1,\ldots,R_r\}$ and, for each $j\leq r$, a positive integer $\ar(R_j)$, called the \emph{arity} of $R_j$.  An {\em $\mathcal{L}$-structure on $M$} is a collection $\mathfrak{M}=(M,\mathcal{R}_1,\ldots,\mathcal{R}_r)$, where $M$ is a set and $\mathcal{R}_j\subseteq M^{\ar(R_j)}$ for each $j\in[1,r]:=\{1,\ldots,r\}$.  When discussing more than one $\mathcal{L}$-structure simultaneously, for example, $\mathfrak{M}$ and $\mathfrak{N}$, we write $R_j^{\mathfrak{M}}$, respectively $R_j^{\mathfrak{N}}$, to denote the interpretation of $R_j$ in $\mathfrak{M}$, respectively $\mathfrak{N}$. We also sometimes call $M$ the {\em universe} of $\mathfrak{M}$ and write $|\mathfrak{M}|=M$.  We write $\mathcal{L}_M$ for the collection of $\mathcal{L}$-structures on $M$.

For convenience, we often abstract away from the particular relations and instead consider a single object which records, for a given finite set $s$, exactly which relations hold of tuples from $s$.  

\begin{definition}
Given an $\mathcal{L}$-structure $\mathfrak{M}=(\Nb,\mathcal{R}_1,\ldots,\mathcal{R}_r)$ and $s\subseteq|\mathfrak{M}|$, we define the {\em quantifier-free type of $s$ in $\mathfrak{M}$} by
\begin{eqnarray*}
\lefteqn{\tp_{\mathfrak{M}}(s)=}\\
&&\{R_j x_1\cdots x_{\ar(R_j)}\mid j\leq r, \{x_1,\ldots,x_{\ar(R_j)}\}\subseteq s\text{ and }\langle x_{1},\ldots,x_{{\ar(R_j)}}\rangle\in \mathcal{R}_j\}.
\end{eqnarray*}
\end{definition}
Our definition of a quantifier-free type is not quite the standard one from model theory, but is slightly simpler and equivalent in our context.  Any structure $\mathfrak{M}$ is determined by the collection of all its quantifier-free types $\{\tp_{\mathfrak{M}}(s)\}_{s\subseteq|\mathfrak{M}|}$.

Every injection $\phi:M'\rightarrow M$ maps $\mathcal{L}_M$ to $\mathcal{L}_{M'}$ in the natural way: $\mathfrak{M}\mapsto\mathfrak{M}^{\phi}:=(M',\mathcal{R}_1^{\phi},\ldots,\mathcal{R}_r^{\phi})$ with 
\[R_jx_1\cdots x_{\ar(R_j)}\in\tp_{\mathfrak{M}^{\phi}}(s)\quad\Longleftrightarrow\quad R_j\phi(x_1)\cdots\phi(x_{\ar(R_j)})\in\tp_{\mathfrak{M}}(\phi(s)),\]
where $\phi(s)$ is the image of $s$ by $\phi$.
We call $\phi$ an \emph{embedding} of $\mathfrak{M}^{\phi}$ into $\mathfrak{M}$, written $\phi:\mathfrak{M}^{\phi}\rightarrow\mathfrak{M}$.  
When $M'\subset M$, the inclusion map, $s\mapsto s$, determines the {\em restriction} $\mathfrak{M}|_{M'}$ of $\mathfrak{M}$ to an $\mathcal{L}$-structure over $M'$.

If $\mu$ is a probability measure on $\mathcal{L}_M$, we write $\mathfrak{X}\sim\mu$ to denote that $\mathfrak{X}$ is a random structure chosen according to $\mu$ and in this case we call $\mathfrak{X}$ a {\em random $\mathcal{L}$-structure on $M$}. 
Random $\mathcal{L}$-structures $\mathfrak{X}$ and $\mathfrak{Y}$ are \emph{equal in distribution}, written $\mathfrak{X}\equalinlaw\mathfrak{Y}$, if $\mathbb{P}(\mathfrak{X}|_S=\mathfrak{S})=\mathbb{P}(\mathfrak{Y}|_S=\mathfrak{S})$  for every $\mathfrak{S}\in\mathcal{L}_S$, for all finite $S\subseteq M$.
A random structure $\mathfrak{X}$ with universe $S$ is \emph{exchangeable} if $\mathfrak{X}=_{\mathcal{D}}\mathfrak{X}^\sigma$ for all permutations $\sigma:S\to S$.

\begin{theorem}[Aldous--Hoover \cite{Aldous1981,Hoover1979}]
Let $\mathfrak{X}$ be an exchangeable random structure with universe $\mathbb{N}$ such that all relations are symmetric with probability $1$.  Then there is a family of Borel measurable functions $\{f_n\}$ such that
 the structure $\mathfrak{Y}$ with quantifier-free types given by
\[\tp_{\mathfrak{Y}}(s)=f_{|s|}((\xi_t)_{t\subseteq s}),\quad s\subseteq\Nb,\]
is equal in distribution to $\mathfrak{X}$, where $\{\xi_t\}_{t\subseteq\Nb}$ are independent and identically distributed (i.i.d.)~Uniform$[0,1]$ random variables.
\end{theorem}

In \cite{CraneTowsner2015}, we introduced the notion of \emph{relatively exchangeable structures}, whose distributions need only be invariant with respect to certain partial automorphisms of a reference structure $\mathfrak{M}$.
\begin{definition}\label{defn:M-exch}
Let $\mathcal{L},\mathcal{L}'$ be signatures and $\mathfrak{M}\in\mathcal{L}_{\mathbb{N}}$.
A random $\mathcal{L}'$-structure $\mathfrak{X}$ with universe $\mathbb{N}$ is \emph{$\mathfrak{M}$-exchangeable}, or {\em relatively exchangeable with respect to $\mathfrak{M}$}, if, for every injection $\phi:S\rightarrow T$ with  $S,T\subseteq\mathbb{N}$ finite, 
  $\mathfrak{M}|_T^\phi=\mathfrak{M}|_S$ implies $\mathfrak{X}|_T^\phi=_{\mathcal{D}}\mathfrak{X}|_S$.
  \end{definition}

Ackerman \cite{Ackerman2015} independently studied a related notion of $\text{Aut}(\mathfrak{M})$-invariant structures, whose distributions are invariant with respect to permutations in the automorphism group of $\mathfrak{M}$.
Though ostensibly a less restrictive condition, $\text{Aut}(\mathfrak{M})$-invariance is identical to $\mathfrak{M}$-exchangeability in the setting of both \cite{Ackerman2015} and \cite{CraneTowsner2015}.
For example, the strongest representations in \cite{Ackerman2015,CraneTowsner2015} hold under the assumptions that $\mathfrak{M}$ is ultrahomogeneous and has $\odap$-DAP, both of which we define below.

A notable departure from $\odap$-DAP occurs when $\mathfrak{M}$ has definable equivalence relations.  
In the absence of $\odap$-DAP, we \cite[Theorem 3.15]{CraneTowsner2015} have proven a more general, albeit weaker, representation in the absence of $\odap$-DAP \cite[Theorem 3.15]{CraneTowsner2015}, and Austin and Panchenko \cite{AustinPanchenko2014} characterized certain hierarchically exchangeable structures, which are defined by an invariance with respect to permutations that preserve certain tree structures.
Although Austin and Panchenko's notion is not defined directly in terms of relative exchangeability, it easily translates as exchangeability relative to some $\mathfrak{M}$ with definable equivalence relations, as we discuss in Section \ref{sec:AP}.  Some definitions and notation are needed before we can state our representation theorem, which appears as Corollary \ref{thm:main_cor}.


\section{Setting and Examples}\label{sec:setting}

\subsection{Fra\"iss\'e Classes}

\begin{definition}
An $\mathcal{L}$-structure  $\mathfrak{M}$ is \emph{ultrahomogeneous} if whenever $S,T\subseteq|\mathfrak{M}|$ are finite and $\pi:\mathfrak{M}|_S\rightarrow\mathfrak{M}|_T$ is an isomorphism, there is an automorphism $\hat\pi$ of $\mathfrak{M}$ such that $\hat\pi\upharpoonright S=\pi$, where $\hat{\pi}\upharpoonright S$ denotes the domain restriction of $\hat{\pi}$ to $S$.
\end{definition}

Throughout the paper, we are interested in the setting of two signatures $\mathcal{L}$ and $\mathcal{L}'$ with $\mathcal{L}$ finite, a fixed ultrahomogeneous $\mathcal{L}$-structure $\mathfrak{M}$ with universe $\mathbb{N}$, and a random $\mathcal{L}'$-structure $\mathfrak{X}$ whose distribution is relatively exchangeable with respect to $\mathfrak{M}$.
It is often more natural to replace $\mathfrak{M}$ with the set of its finite embedded substructures, called the {\em age of $\mathfrak{M}$} and denoted $\age(\mathfrak{M})$.  To this end we recall some basic facts about Fra\"iss\'e classes and introduce the idea of a random process over a Fra\"iss\'e class.

\begin{definition}
The {\em age of $\mathfrak{M}$}, denoted $\age(\mathfrak{M})$, is the set of finite structures $\mathfrak{S}$ such that there is an embedding of $\mathfrak{S}$ into $\mathfrak{M}$.

  A collection $\mathcal{K}$ of finite structures has the \emph{amalgamation property} if whenever $\mathfrak{S},\mathfrak{T}_0,\mathfrak{T}_1\in\mathcal{K}$ and $f_i:\mathfrak{S}\rightarrow\mathfrak{T}_i$ are embeddings, there is a $\mathfrak{U}\in\mathcal{K}$ and embeddings $g_i:\mathfrak{T}_i\rightarrow\mathfrak{U}$ with $g_0\circ f_0=g_1\circ f_1$.
  
  A collection $\mathcal{K}$ of finite structures is a \emph{Fra\"iss\'e class} if it has the amalgamation property and is closed under isomorphism and taking substructures, that is, $\mathfrak{S}\in\mathcal{K}$ implies $\mathfrak{S}|_S\in\mathcal{K}$ for all $S\subseteq|\mathfrak{S}|$.
\end{definition}

In a finite relational signature, Fra\"iss\'e's theorem \cite[Theorem 6.1.2]{HodgesShorter} says that, on the one hand, if $\mathfrak{M}$ is a countable ultrahomogeneous structure then $\age(\mathfrak{M})$ is a Fra\"iss\'e class, and on the other hand, every Fra\"iss\'e class is the age of a unique (up to isomorphism) countable ultrahomogeneous structure, called the \textit{Fra\"iss\'e limit} of the class.
Thus, when $\mathfrak{M}$ is ultrahomogeneous, our notion of $\mathfrak{M}$-exchangeability from Definition \ref{defn:M-exch} and \cite[Definition 2.3]{CraneTowsner2015} is equivalent to $\age(\mathfrak{M})$-exchangeability in the sense of the following definition.


\begin{definition}\label{defn:rel exch}

  Let $\mathcal{K}$ be a Fra\"iss\'e class of $\mathcal{L}$-structures.  A \emph{$\mathcal{K}$-exchangeable $\mathcal{L}'$-structure} is a collection of random $\mathcal{L}'$-structures $\{\mathfrak{X}(\mathfrak{S})\}_{\mathfrak{S}\in\mathcal{K}}$ so that each $\mathfrak{X}(\mathfrak{S})$ is a random $\mathcal{L}'$-structure on $|\mathfrak{S}|$ and whenever $\pi:\mathfrak{S}\rightarrow\mathfrak{T}$, $\mathfrak{S},\mathfrak{T}\in\mathcal{K}$, is an embedding, $(\mathfrak{X}(\mathfrak{T}))^\pi=_{\mathcal{D}}\mathfrak{X}(\mathfrak{S})$.

  Let $\mathfrak{M}$ be ultrahomogeneous.  A random $\mathcal{L}'$-structure $\mathfrak{X}$ on $|\mathfrak{M}|$ is $\mathfrak{M}$-exchangeable if whenever $S,T\subseteq|\mathfrak{M}|$ are finite and $\pi:\mathfrak{M}|_S\rightarrow\mathfrak{M}|_T$ is an embedding, $\mathfrak{X}|^\phi_T=_{\mathcal{D}}\mathfrak{X}|_S$.
\end{definition}

The notion of $\mathcal{K}$-exchangeability, though equivalent to $\mathfrak{M}$-exchangeability in a certain sense, is often easy to work with.
We now show that these two notions are equivalent.

Let $\mathfrak{M}$ be a countable ultrahomogeneous structure with Fra\"iss\'e class $\mathcal{K}=\age(\mathfrak{M})$.  If $\mathfrak{S}\in\mathcal{K}$ and $\mathfrak{M}$ are both well-ordered---say, because $|\mathfrak{S}|=n$ and $|\mathfrak{M}|=\mathbb{N}$, or because we have otherwise chosen a well-ordering $<$ on $|\mathfrak{S}|$---then there is a canonical embedding $\rho_{\mathfrak{S},<}:\mathfrak{S}\rightarrow\mathfrak{M}$ given by inductively choosing $\rho_{\mathfrak{S},<}(i)=m_i$, $i\geq1$, so that 
$m_i$ is smallest positive integer larger than $m_{i-1}$ such that $\rho_{\mathfrak{S},<}\upharpoonright\{j: j\leq i\}$ is an embedding of $\mathfrak{S}|_{\{j: j\leq i\}}$ into $\mathfrak{M}$.  (Note that this definition \emph{does} depend on the choice of ordering $<$.)

Given an $\mathfrak{M}$-exchangeable structure $\mathfrak{X}$, we can define a $\mathcal{K}$-exchangeable structure $\{\mathfrak{X}(\mathfrak{S})\}_{\mathfrak{S}\in\mathcal{K}}$ by putting $\mathfrak{X}(\mathfrak{S})=\mathfrak{X}^{\rho_{\mathfrak{S},<}}$ for any ordering $<$ since $\mathfrak{M}$-exchangeability ensures that the distribution of $\mathfrak{X}(\mathfrak{S})$ is independent of the choice of ordering.
Conversely, given a $\mathcal{K}$-exchangeable structure $\{\mathfrak{X}(\mathfrak{S})\}_{\mathfrak{S}\in\mathcal{K}}$, we construct an $\mathfrak{M}$-exchangeable structure $\mathfrak{X}$ by randomly choosing $\mathfrak{X}|_{[n]}$, for each $n\geq1$, according to the conditional distribution of $\mathfrak{X}(\mathfrak{M}|_{[n]})$ given $\mathfrak{X}(\mathfrak{M}|_{[n-1]})$.  Here $\mathcal{K}$-exchangeability ensures that $\mathfrak{X}^{\pi}=_{\mathcal{D}}\mathfrak{X}$ for every automorphism $\pi$ of $\mathfrak{M}$.
For the remainder of the paper, we primarily consider $\mathcal{K}$-exchangeable structures, with the understanding that these are equivalent to $\mathfrak{M}$-exchangeable structures for the appropriate choice of $\mathfrak{M}$.

The na\"ive generalization of Aldous--Hoover to $\mathcal{K}$-exchangeable structures suggests an analogous representation with the additional input $\mathfrak{S}$, that is, a representation of the form
\begin{equation}
\tp_{\mathfrak{Y}}(s)=f_{|s|}(\mathfrak{S}|_{s},(\xi_t)_{t\subseteq s}),\quad s\subseteq\Nb,\label{eq:DAP_rep}
\end{equation}
for $(\xi_t)_{t\subseteq\Nb}$ i.i.d.\ Uniform$[0,1]$ random variables.
(Recall that we are restricting to the symmetric case for now.   The representation for structures with  possibly asymmetric relations is slightly more involved.)

  In \cite{CraneTowsner2015}, we noted that the representation in \eqref{eq:DAP_rep} is not always possible: even though the \emph{distribution} of $\mathfrak{X}(\mathfrak{S})$ does not depend on an ordering of $|\mathfrak{S}|$, it may be that any \emph{representation} does.
 The representation in \eqref{eq:DAP_rep} holds when $\mathcal{K}$ satisfies the additional restriction of $\odap$-DAP, as proven independently in \cite[Theorem 3.2]{CraneTowsner2015} and \cite{Ackerman2015}.

\begin{definition}
  Suppose that for each $i\leq n$, $\mathfrak{S}_i$ is an $\mathcal{L}$-structure with $|\mathfrak{S}_i|=[n]\setminus\{i\}$.  We say $(\mathfrak{S}_i)_{1\leq i\leq n}$ is an \emph{amalgamation plan of size $n$} if $\mathfrak{S}_i|_{[n]\setminus\{i,j\}}=\mathfrak{S}_j|_{[n]\setminus\{i,j\}}$ for all $1\leq i,j\leq n$.

  An \emph{amalgam} of such an amalgamation plan is a structure $\mathfrak{S}$ with $|\mathfrak{S}|=[n]$ and $\mathfrak{S}|_{[n]\setminus\{i\}}=\mathfrak{S}_i$ for all $i\in[n]$.

  We say a collection $\mathcal{K}$ of structures \emph{has the $n$-disjoint amalgamation property} ($n$-DAP) if, for every amalgamation plan $(\mathfrak{S}_i)_{1\leq i\leq n}$ of size $n$ with each $\mathfrak{S}_i\in \mathcal{K}$, there is an amalgam $\mathfrak{S}\in \mathcal{K}$.  We say $\mathcal{K}$ has {\em $\odap$-DAP} if $\mathcal{K}$ has $n$-DAP for all $n\geq1$.
\end{definition}

The main theorem in \cite{AFP} suggests a strong relationship between Aldous--Hoover-type representations and $2$-DAP, so our interest here is the case where $\mathcal{K}$ has $2$-DAP but not $\odap$-DAP.
 
\subsection{Equivalence Relations with Infinitely Many Classes}
One common obstacle to $n$-DAP is the presence of definable equivalence relations.  

\begin{example}\label{ex:one_eq}
Consider a signature $\mathcal{L}$ with a single binary relation $R$ and let $\mathcal{K}$ be the collection of all $\mathcal{L}$-structures $\mathfrak{S}=(S,R^{\mathfrak{S}})$, with $S\subset\Nb$ finite, such that $R^{\mathfrak{S}}$ is an equivalence relation.  Then $\mathcal{K}$ fails to have $3$-DAP.  For example, let $\mathfrak{S}_1,\mathfrak{S}_2,\mathfrak{S}_3\in\mathcal{K}$ be structures so that $(2,3)\in R^{\mathfrak{S}_1}$, $(1,3)\in R^{\mathfrak{S}_2}$, and $(1,2)\not\in R^{\mathfrak{S}_3}$.  These three structures constitute an amalgamation plan of size $3$, but they have no amalgam.

The absence of $\odap$-DAP obstructs the representation \eqref{eq:DAP_rep}.
For suppose $\mathcal{L}'$ is a signature with a single unary relation $P$, $\mathfrak{M}$ is an $\mathcal{L}$-structure, and, for every $\mathfrak{S}\in\mathcal{K}$, $\mathfrak{X}(\mathfrak{S})$ is obtained by independently including or excluding each equivalence class of $\mathfrak{S}$ in $P^{\mathfrak{X}(\mathfrak{S})}$ with probability $1/2$.
Although $\mathfrak{X}=\{\mathfrak{X}(\mathfrak{S})\}_{\mathfrak{S}\in\mathcal{K}}$ is $\mathcal{K}$-exchangeable, $\mathfrak{S}|_{\{x\}}$ is trivial for every $x\in\mathbb{N}$, preventing any way to coordinate elements in the same equivalence class in any representation of the form
\begin{equation}\label{eq:failed rep}x\in P^{\mathfrak{X}(\mathfrak{S})}\quad\Longleftrightarrow\quad Px\in f(\mathfrak{S}|_{\{x\}},\xi_\emptyset, \xi_x)\end{equation}
for some measurable function $f$ and i.i.d.\ Uniform$[0,1]$ random variables $\{\xi_{\emptyset}; (\xi_{x})_{x\in\mathbb{N}}\}$.
Indeed, the first argument of $f(\mathfrak{S}|_{\{x\}},\xi_{\emptyset},\xi_x)$ in \eqref{eq:failed rep} is moot, making the construction on the righthand side exchangeable even though $\mathfrak{X}(\mathfrak{S})$ need only be $\mathfrak{S}$-exchangeable.
\end{example}

We have previously \cite[Theorem 3.15]{CraneTowsner2015} proven a general representation for random structures exchangeable relative to any Fra\"iss\'e class $\mathcal{K}$, but in that representation $\tp_{\mathfrak{X}(\mathfrak{S})}(s)$ depends on the whole ordered substructure $\mathfrak{S}|_{[1,\max s]}$ rather than just $\mathfrak{S}|_{s}$.  When  definable equivalence relations are the only obstruction to $n$-DAP, we can give a more natural representation that accounts for the structure of these relations.

By its description, it is clear that the structure $\mathfrak{X}$ in Example \ref{ex:one_eq} can be represented in the form
\[x\in P^{\mathfrak{X}(\mathfrak{S})}\quad\Longleftrightarrow\quad Px\in f(\mathfrak{S}|_{\{x\}},\xi_\emptyset, \xi_x,\xi_{[x]_{\mathfrak{S}}}),\]
where $\xi_{[x]}$ is a Uniform$[0,1]$ random variable indexed by the \emph{equivalence class} $[x]_{\mathfrak{S}}$ of $x\in\mathbb{N}$ in $\mathfrak{S}$.
In fact, $f$ in this example can be chosen to depend only on $\xi_{[x]_{\mathfrak{S}}}$.

A further complication arises when we allow multiple equivalence relations.
\begin{example}\label{ex:two_eq}
  Consider a signature $\mathcal{L}$ with two binary relations $R, S$.  Let $\mathcal{K}_1$ be the collection of $\mathcal{L}$-structures for which both $R$ and $S$ are equivalence relations.  Let $\mathcal{K}_2\subseteq\mathcal{K}_1$ be those structures $\mathfrak{S}$ in which $(x,y)\in S^{\mathfrak{S}}$ implies $(x,y)\in R^{\mathfrak{S}}$.

For $\mathcal{K}_1$, we have to account for various combinations of equivalence classes from the two relations.  We expect to have dependencies not only on $y$, $[y]_{R^{\mathfrak{S}}}$, and $[y]_{S^{\mathfrak{S}}}$, but also on $[y]_{R^{\mathfrak{S}}}\cap [y]_{S^{\mathfrak{S}}}$.  For bookkeeping purposes, we write $E(y)$ for the set of equivalence classes of $y$ including under equality\footnote{Strictly speaking, $E(y)$ should be the set of equivalence classes of $y$ in the Fra\"iss\'e limit, or, equivalently, in sufficiently large structures in $\mathcal{K}_1$.  It is possible that a particular $\mathfrak{S}$ is ``too small'', in the sense that $[y]_{R^{\mathfrak{S}}}$ happens to equal $[y]_{S^{\mathfrak{S}}}$, even though the equivalence classes will differ in larger structures; we ignore this complication here and address it properly in our formal definition below.}, so that  $[y]_{=}=\{y\}$ and $E(y)=\{[y]_{=},[y]_{R^{\mathfrak{S}}},[y]_{S^{\mathfrak{S}}}\}$.   Our representation depends on the antichains from $E(y)$.  A \emph{blur} of $y$ (in $\mathcal{K}_1$) is any antichain from $E(y)$.  We write $B(y)$ to denote the set of blurs, so in this case $B(y)=\{\{[y]_{=}\},\{[y]_{R^{\mathfrak{S}}},[y]_{S^{\mathfrak{S}}}\},\{[y]_{R^{\mathfrak{S}}}\},\{[y]_{S^{\mathfrak{S}}}\}\}$.  We will show that the representation for $\mathcal{K}_1$-exchangeable structures $\mathfrak{Y}\equalinlaw\mathfrak{X}(\mathfrak{S})$ with unary relations has the form
\[\tp_{\mathfrak{Y}}(y)=f_{1}(\mathfrak{S}|_{y},(\xi_t)_{t\in B(y)}).\]

The class $\mathcal{K}_2$ contains structures with two nested equivalence relations.  Here it no longer makes sense to consider $[y]_{R^{\mathfrak{S}}}$ and $[y]_{S^{\mathfrak{S}}}$ separately, as reflected in the new definition of $B(y)=\{\{[y]_{=}\},\{[y]_{R^{\mathfrak{S}}}\},\{[y]_{S^{\mathfrak{S}}}\}\}$, which excludes $\{[y]_{R^{\mathfrak{S}}}, [y]_{S^{\mathfrak{S}}}\}$ because this is no longer an antichain of $E(y)$.  We will see that any symmetric $\mathcal{K}_2$-exchangeable structure has a representation of the form
\[\tp_{\mathfrak{Y}}(y)=f_{1}(\mathfrak{S}|_{y},(\xi_t)_{t\in B(y)})\]
with the adjusted definition of the blur.

\end{example}

For more general collections $\mathcal{K}$ and $\mathcal{K}$-exchangeable structures $\mathfrak{Y}$, we define $\tp_{\mathfrak{Y}}(s)=f_{|s|}(\cdots)$ with the random variables $\xi_t$ in the representation ranging over antichains from $\bigcup_{y\in s}E(y)$.

\subsection{Equivalence Relations with Finitely Many Classes}

All of the equivalence relations considered above have unboundedly many equivalence classes in $\mathcal{K}$ (or, equivalently, infinitely many equivalence classes in the Fra\"iss\'e limit).  We also wish to consider equivalence relations with a bounded number of classes.
These structures are easier to handle than those with infinitely many equivalence classes, even though they introduce a further obstacle to $\odap$-DAP.
Namely, in addition to avoiding amalgamation plans which contradict being an equivalence relation, such classes also avoid amalgamation plans which would amalgamate to have too many classes.

\begin{example}\label{ex:one_eq_two_classes}
  Consider a signature $\mathcal{L}$ with a single binary relation $R$ and the collection $\mathcal{K}$ of all $\mathcal{L}$-structures $\mathfrak{S}$ for which $R^{\mathfrak{S}}$ is an equivalence relation with at most two equivalence classes.

For example, with $\mathcal{K}$ as in Example \ref{ex:one_eq_two_classes}, a representation should have the form
\[\tp_{\mathfrak{X}(\mathfrak{S})}(s)=f_{|s|}(\mathfrak{S}|_{s},(\xi_t)_{t\subseteq s},\eta\upharpoonright\mathfrak{S}|_s),\quad s\subseteq\Nb,\]
where $\eta:\mathfrak{S}/R^{\mathfrak{S}}\rightarrow\{0,1\}$ is a randomly chosen function that labels the equivalence classes of $\mathfrak{S}$ uniformly without replacement from $\{0,1\}$ and $\eta\upharpoonright\mathfrak{S}|_s$ is the restriction of $\eta$ to a labeling of the equivalence classes of $\mathfrak{S}|_s$.  Furthermore, we should require that the distribution is symmetric in the choice of $\eta$ in the sense that if $\eta,\eta'$ are two different ways of labeling the same equivalence classes, then the expected value of $f$ over different values of $(\xi_t)_{t\subseteq s}$ is the same.

For example, suppose that $\mathcal{L}'$ has a single unary relation $P$ and $\mathfrak{X}(\mathfrak{S})$ is obtained by choosing one equivalence class of $\mathfrak{S}$ uniformly at random, putting the chosen class into $P^{\mathfrak{X}(\mathfrak{S})}$ and the unchosen class outside of $P^{\mathfrak{X}(\mathfrak{S})}$.  
Our representation should have the form
\[x\in P^{\mathfrak{X}(\mathfrak{S})}\quad\Longleftrightarrow\quad Px\in f(\mathfrak{S}|_{\{x\}},\xi_\emptyset,\xi_{\{x\}},\eta([x]_\sim)),\quad x\in|\mathfrak{S}|,\]
where $[x]_\sim$ denotes the equivalence class containing $x$ in $\mathfrak{S}$.
A tempting choice would be 
\[Px\in f(\mathfrak{S}|_{\{x\}},\xi_\emptyset,\xi_{\{x\}},\eta([x]_\sim))\quad\Longleftrightarrow\quad \eta([x]_\sim)=1,\]
allowing $\eta$ to determine which equivalence class is in $P^{\mathfrak{X}(\mathfrak{S}})$.
This choice, however, violates our symmetry criterion.  We prefer a representation where $\eta$ labels classes but $\xi_\emptyset$ determines which class belongs to $P$, for example,
\[Px\in f(\mathfrak{S}|_{\{x\}},\xi_{\emptyset},\xi_{\{x\}},\eta([x]_\sim))\quad\Longleftrightarrow\quad
\begin{array}{l}
\eta([x]_\sim)=1\text{ and }\xi_{\emptyset}\in[0,1/2]\text{ or}\\
\eta([x]_\sim)=0\text{ and }\xi_{\emptyset}\in(1/2,1].
\end{array}\]

\end{example}

\subsection{Asymmetric Structures}\label{section:asymmetric}

The above discussion specializes to the situation in which all relations in $\mathfrak{X}$ are assumed to be symmetric.  Allowing asymmetric relations introduces a new complication: for instance, when determining the type of $\{x,y\}$ in $\mathfrak{X}$, we may have to decide that exactly one of $(x,y)$ and $(y,x)$ belongs to some relation.  Various approaches appear in the literature \cite{AroskarCummings2014,MR2463439,KallenbergSymmetries}.

In \cite{CraneTowsner2015} we introduced a different approach which we believe best captures the philosophy of the Aldous--Hoover theorem.  When we need random data to represent the joint behavior of the pair $\{x,y\}$, we divide it into a symmetric part $\xi_{\{x,y\}}$ and an asymmetric part $\prec_{(x,y)}$, parallel to the way the Aldous--Hoover Theorem divides the behavior of the pair $\{x,y\}$ into $\xi_x,\xi_y$, and $\xi_{\{x,y\}}$.

\begin{definition}
  A \emph{uniform random ordering} of a finite set $s$ is an ordering $\prec_s$ of $s$ chosen uniformly at random.  
\end{definition}

Here we need to generalize the presence of random orderings to equivalence classes in the set of blurs $B(s)$, so our representations will include the data of a uniform random ordering $\prec_\tau$ for each $\tau\in B(s)$.

\subsection{Full Representation}

For our full result we consider classes $\mathcal{K}$ with multiple equivalence relations.  Inevitably, there is a trade-off between complexity and generality.  For now we quickly describe the collections of equivalence relations covered in our main theorem, delaying a careful definition to Section \ref{sec:sequences}.
\begin{itemize}
\item We assume that $\mathcal{K}$ is a Fra\"iss\'e class with $2$-DAP, $\sim_1,\ldots,\sim_d$ is a \emph{strongly orderly} sequence of equivalence relations in $\mathcal{K}$, and $\mathcal{K}$ has $\odap$-DAP \emph{up to $\sim_1,\ldots,\sim_d$}.  Strongly orderly (Definition \ref{defn:strongly orderly}) is a technical condition ruling out several pathological cases which would require an even more complicated representation, while $\odap$-DAP up to $\sim_1,\ldots,\sim_d$ (Definition \ref{defn:upto}) says that these equivalence relations constitute the only obstacles to $\odap$-DAP.
\item We write $\#(\sim_r)\in\mathbb{N}\cup\{\infty\}$ for the number of equivalence classes of $\sim_r$ (in a certain sense defined precisely below).
\item Each $\sim_r$ is an equivalence relation on $k_r$-tuples, and if $\#(\sim_r)=\infty$ then $k_r=1$.
\item For each $\mathfrak{S}\in\mathcal{K}$ and each $r$, we write $V_{r,\mathfrak{S}}$ for the domain on which $\sim_r$ is defined and let $U_{r,\mathfrak{S}}=V_{r,\mathfrak{S}}/\sim_r$ be the collection of equivalence classes of $\sim_r$.  
\item For any $y\in|\mathfrak{S}|$ with $\mathfrak{S}\in\mathcal{K}$, we write $E(y)$ for the collection of pairs $(y,\sim_r)$, where $\#(\sim_r)=\infty$, together with the pair $(y,=)$.  If $s\subseteq|\mathfrak{S}|$, we write $E(s)=\bigcup_{y\in s}E(y)/\simeq$, where $(y,\sim)\simeq(y',\sim)$ if and only if $y\sim y'$.  We define an ordering on $E(s)$ by $(y,\sim)\leq (y',\sim')$ if for some (equivalently, any) embedding $\gamma:\mathfrak{S}\rightarrow\mathfrak{M}$, $\mathfrak{M}$ is the Fra\"iss\'e limit of $\mathcal{K}$, $[y]_\sim\subseteq[y']_{\sim'}$.  We define $B(s)$ to consist of all antichains (under $\leq$) in $E(s)$.
\end{itemize}

We often identify the element $(y,\sim)$ of $E(y)$ with the equivalence class $[y]_\sim$.  Taken literally, this causes minor technical problems in the case where $\mathfrak{S}$ is small enough that $[y]_\sim$ and $[y]_{\sim'}$ coincide in $\mathfrak{S}$ even though they are not the same set in larger structures.

\begin{definition}
If $\{f_n\}$ is a set of measurable functions, the \emph{canonical exchangeable process $\mathfrak{X}^{\{f_n\}}$ generated by the $\{f_n\}$} is defined for each $\mathfrak{S}\in\mathcal{K}$ by choosing
 \begin{equation}\label{eq:bits}
\begin{array}{l}
  \bullet\quad \xi_\tau\text{ uniformly in }[0,1]\text{ and }\prec_\tau\text{ a uniform random ordering of }\tau,\text{ for each }\\
  \text{finite set }t\subseteq |\mathfrak{S}|\text{ and each }\tau\in B(t),\text{ and}\\
  \bullet\quad
  \eta_r:U_{r,\mathfrak{S}}\rightarrow[\#(\sim_r)],\text{ a uniform random injective function for each }r\in[d]\\\text{ with }\#(\sim_r)\neq\infty,
  \end{array}
  \end{equation}
with all random variables chosen independently, and setting
\[\tp_{\mathfrak{X}^{\{f_n\}}(\mathfrak{S})}(s)=f_{|s|}(\mathfrak{S}|_{s},(\xi_\tau)_{\tau \in B(s)},(\prec_{t})_{t \subseteq s},(\eta_r\upharpoonright \mathfrak{S}|_s)_{r\mid \#(\sim_r)\neq\infty}),\]
where $\eta_r\upharpoonright\mathfrak{S}|_s$ denotes the domain restriction of $\eta_r$ to $V_{r,\mathfrak{S}|_s}/\sim_r$.

We say $\{f_n\}$, or equivalently $\mathfrak{X}^{\{f_n\}}$, is \emph{eq-symmetric} if the distribution of $\mathfrak{X}^{\{f_n\}}(\mathfrak{S})$ is independent of the variables $\eta_r$, that is, if for any $\mathfrak{S},\mathfrak{T}$,
\[\mathbb{P}(\mathfrak{X}^{\{f_n\}}(\mathfrak{S})=\mathfrak{T}\mid \{\eta_r\}_{r\mid \#(\sim_r)\neq\infty})=\mathbb{P}(\mathfrak{X}^{\{f_n\}}(\mathfrak{S})=\mathfrak{T}).\]
\end{definition}

We now state our main theorem.
\begin{theorem}\label{thm:main}
Let $\mathcal{K}$ be a Fra\"iss\'e class with $2$-DAP and with $\odap$-DAP up to a strongly orderly sequence of definable equivalence relations $\sim_1,\ldots,\sim_d$.  Let $\mathfrak{X}$ be $\mathcal{K}$-exchangeable.
Then there exist eq-symmetric functions $\{f_n\}$ such that $\mathfrak{X}^{\{f_n\}}=_{\mathcal{D}}\mathfrak{X}$.
\end{theorem}

\begin{cor}\label{thm:main_cor}
  Let $\mathfrak{M}$ be ultrahomogeneous and suppose $\age(\mathfrak{M})$ satisfies $2$-DAP and $\odap$-DAP up to a strongly orderly sequence of definable equivalence relations $\sim_1,\ldots,\sim_d$.  Let $\mathfrak{X}$ be $\mathfrak{M}$-exchangeable.
Then there are eq-symmetric functions $\{f_n\}$ such that $\mathfrak{X}^{\{f_n\}}\equalinlaw\mathfrak{X}$, for $\xi_\tau,\prec_\tau$, and  $\eta_r$ chosen as in \eqref{eq:bits} and setting
\[\tp_{\mathfrak{X}^{\{f_n\}}}(s)=f_{|s|}(\mathfrak{M}|_{s},(\xi_\tau)_{\tau \in B(s)},(\prec_{t})_{t \subseteq s},(\eta_r\upharpoonright \mathfrak{M}|_s)_{r\mid\#(\sim_r)\neq\infty})\]
for each $s\subseteq\Nb$.
\end{cor}

\subsection{Application to Austin--Panchenko representation}\label{sec:AP}

Although we have described our setting in terms of random $\mathcal{L}$-structures, our framework can also handle $\mathfrak{M}$-exchangeable real valued random structures.  We take a signature $\mathcal{L}'$ consisting of countably many unary symbols $U_1,U_2,\ldots$ and when $\mathfrak{S}$ is an $\mathcal{L}'$-structure, we assign a real value to each $x\in|\mathfrak{S}|$ by $v(x)=\sum_i 2^{-i}\chi_{U_i^{\mathfrak{S}}}(x)$, where $\chi_{U_i^{\mathfrak{S}}}(x)=1$ if $x\in U_i^{\mathfrak{S}}$ and $\chi_{U_i^{\mathfrak{S}}}(x)=0$ otherwise.  Conversely, if $v:S\rightarrow[0,1]$, we define a structure $\mathfrak{S}$ with $|\mathfrak{S}|=S$ by putting $x\in U_i^{\mathfrak{S}}$ if and only if the $i$-th bit of the binary expansion of $v(x)$ equals $1$.  (Note that we can replace $[0,1]$ with any standard Borel space.)

As an example of our general framework, we now show how it generalizes the results of \cite{AustinPanchenko2014}.  We first consider arrays of random variables $(X_\alpha)_{\alpha\in\mathbb{N}^r}$ with values from $[0,1]$.  We will consider $\mathbb{N}^r$ as an $\mathcal{L}$-structure (the exact structure will be described below), so that, using the previous paragraph, we can define a random $\mathcal{L}'$-structure $\mathfrak{X}$ on $\mathbb{N}^r$ where $v(\mathfrak{X}(\mathbb{N}^r|_{\{\alpha\}}))=X_\alpha$.  We consider a class of permutations $H_r$ consisting of those permutations $\pi:\mathbb{N}^r\rightarrow\mathbb{N}^r$ which preserve initial segments; that is, if $\alpha$ and $\beta$ first differ at the $k$-th place then $\pi(\alpha)$ and $\pi(\beta)$ also first differ at the $k$-th place.

In our notation, \cite{AustinPanchenko2014} shows that if 
\[(\mathfrak{X}(\mathbb{N}^r|_{\{\pi(\alpha)\}}))_{\alpha\in\mathbb{N}^r}=_{\mathcal{D}}(\mathfrak{X}(\mathbb{N}^r|_{\{\alpha\}}))_{\alpha\in\mathbb{N}^r}\]
for each $\pi\in H_r$ then $\mathfrak{X}$ has a representation of the form
\[\tp_{\mathfrak{X}}(\alpha)=f((\xi_\beta)_{\beta\sqsubseteq\alpha}),\]
where $\beta\sqsubseteq\alpha$ means $\beta$ is an initial segment of $\alpha$.

To obtain this as a consequence of Corollary \ref{thm:main_cor}, we take an $\mathcal{L}$-structure $\mathfrak{M}$ with $|\mathfrak{M}|=\mathbb{N}^r$, where $\mathcal{L}=\{R_1,\ldots,R_{r-1}\}$.  We define $R_i^{\mathfrak{M}}=\{(\alpha,\beta)\in|\mathfrak{M}|^2\mid \forall j<i\ \alpha_j=\beta_j\}$, where $\alpha=\langle \alpha_0,\ldots,\alpha_{r-1}\rangle$, $\beta=\langle \beta_0,\ldots,\beta_{r-1}\rangle$.  It is easy to see that $H_r$ coincides with the automorphism group of $\mathfrak{M}$.

Note also that $\mathfrak{M}$ has $\odap$-DAP up to $\sim_{R_1},\ldots,\sim_{R_{r-1}}$, with $\sim_{R_{i+1}}^*=\sim_{R_i}$ and $\#(\sim_{R_i})=\infty$ for each $i$.  Our main theorem gives the representation
\[\tp_{\mathfrak{X}}(\alpha)=f((\xi_\tau)_{\tau\in B(\alpha)}).\]
It therefore suffices to observe the correspondence between $B(\alpha)=\{\alpha\}\cup\{[\alpha]_{\sim_{R_i}}\}_{i=1,\ldots,r-1}$ and $\{\beta\mid \beta\sqsubseteq\alpha\}$: each equivalence class $[\alpha]_{\sim_{R_i}}$ is determined by the initial segment $\langle \alpha_0,\ldots,\alpha_{i-1}\rangle\sqsubset\alpha$.

Austin and Panchenko \cite{AustinPanchenko2014} generalize this representation to random arrays $(X_\alpha)_{\alpha\in\mathbb{N}^{r_1}\times\cdots\times\mathbb{N}^{r_k}}$ with the property that
\[(X_{(\pi_1(\alpha_1),\ldots,\pi_k(\alpha_k))})=_{\mathcal{D}}(X_{(\alpha_1,\ldots,\alpha_k)})\]
when each $\pi_i\in H_{r_i}$.  They show that such arrays have a representation of the form
\[\tp_{\mathfrak{X}}(\alpha_1,\ldots,\alpha_k)=f((\xi_{\beta_1,\ldots,\beta_k})_{\beta_1\sqsubseteq\alpha_1,\ldots,\beta_k\sqsubseteq\alpha_k}).\]

We can obtain this as a consequence of Corollary \ref{thm:main_cor} with a bit of care.  We take a structure $\mathfrak{M}$ with $|\mathfrak{M}|=\mathbb{N}^{r_1}\times\cdots\mathbb{N}^{r_k}\times\mathbb{N}$, where $\mathcal{L}=\{R_{1,1},\ldots,R_{1,r_1},R_{2,1},\ldots,R_{2,r_2},\ldots,R_{k,1},\ldots,R_{k,r_k}\}$ and we define 
\[R_{m,i}^{\mathfrak{M}}=\{((\alpha_1,\ldots,\alpha_k,\alpha_+),(\beta_1,\ldots,\beta_k,\beta_+))\mid \forall j<i\ (\alpha_m)_i=(\beta_m)_i\}.\]
(Recall that $(\alpha_m)_i$ means the $i$-th coordinate of the sequence $\alpha_m$.)
We see that the automorphisms of $\mathfrak{M}$ are the maps $(\pi_1,\ldots,\pi_k,\pi_{+})$ so that each $\pi_i\in H_{r_i}$ and $\pi_{+}$ is an arbitrary permutation of $\mathbb{N}$.  $\mathfrak{M}$ has $\odap$-DAP up to $\sim_{R_{1,1}},\ldots,\sim_{R_{k,r_k-1}}$ where $\sim_{R_{i,j+1}}^*=\sim_{R_{i,j}}$, $\sim_{R_{i,1}}^*$ is the equivalence relation with only one class, and $\#(\sim_{R_{i,j}})=\infty$ for all $i,j$.  

The main oddity here is the extra coefficient $\alpha_+$, which we will ultimately need to eliminate by extending the random variables $X_{(\alpha_1,\ldots,\alpha_k)}$ to an $\mathcal{L}'$-structure on $\mathfrak{M}$ so that $v(\tp_{\mathfrak{X}}(\alpha_1,\ldots,\alpha_k,\alpha_+))=X_{(\alpha_1,\ldots,\alpha_k)}$, ignoring $\alpha_+$.

Our main theorem gives the representation
\[\tp_{\mathfrak{X}}((\alpha_1,\ldots,\alpha_k,\alpha_+))=f((\xi_\tau)_{\tau\in B(\alpha_1,\ldots,\alpha_k,\alpha_+)}).\]
Since $\mathfrak{X}$ has the property that $\tp_{\mathfrak{X}}((\alpha_1,\ldots,\alpha_k,\alpha_+))=\tp_{\mathfrak{X}}((\alpha_1,\ldots,\alpha_k,\alpha'_+))$, we may further assume that $f((\xi_\tau)_{\tau\in B(\alpha_1,\ldots,\alpha_k,\alpha_+)})$ depends only on those $\tau$ which are independent of $\alpha_+$.

We must identify $B(\alpha)$ where $\alpha=\langle\alpha_1,\ldots,\alpha_k,\alpha_+\rangle$.  First, note that $E(\alpha)$ consists of $[\alpha]_=$ together with each of the equivalence relations $[\alpha]_{R_{i,j}}$.  The only element of $E(\alpha)$ depending on $\alpha_+$ is the singleton $\{[\alpha]_=\}$, so we should obtain an identification of $B(\alpha)\setminus\{\{[\alpha]_=\}\}$ with $\{\langle\beta_1,\ldots,\beta_k\rangle\mid \beta_1\sqsubseteq\alpha_1,\ldots,\beta_k\sqsubseteq\alpha_k\}$.

  Since $[\alpha]_{R_{i,j}}\subseteq[\alpha]_{R_{i,j'}}$ whenever $j>j'$, any antichain in $E(\alpha)$ contains at most one $[\alpha]_{R_{i,j}}$ for each $i$.  We can identify the choice $[\alpha]_{R_{i,j}}$ with the initial segment $\langle(\alpha_i)_0,\ldots,(\alpha_i)_{j-1}\rangle$ of $\alpha_i$.  If some antichain $\tau\in B(\alpha)\setminus\{\{[\alpha]_=\}\}$ does not contain any $[\alpha]_{R_{i,j}}$ for some $i$, we can identify this with the empty initial segment of $\alpha_i$.

So given $\tau\in B(\alpha)\setminus\{\{[\alpha]_=\}\}$, we define 
\[\beta^\tau_i=\left\{\begin{array}{cc}
\langle(\alpha_i)_0,\ldots,(\alpha_i)_{j-1}\rangle, & [\alpha]_{R_{i,j}}\in\tau,\\
\emptyset,& \text{otherwise}.
\end{array}\right.\]  
This gives a bijection $\tau\mapsto\langle\beta^\tau_1,\ldots,\beta^\tau_k\rangle$ between $B(\alpha)\setminus\{\{[\alpha]_=\}\}$ and $\{\langle\beta_1,\ldots,\beta_k\rangle\mid \beta_1\sqsubseteq\alpha_1,\ldots,\beta_k\sqsubseteq\alpha_k\}$, showing that the two representations are identical.

\section{Sequences of Equivalence Relations}\label{sec:sequences}

\subsection{Explicit Formulas}

\begin{definition}
A pair of formulas $\phi(\vec x,\vec y)$ and $\psi(\vec x)$, where $|\vec x|=|\vec y|=n$, with exactly the displayed free variables and no parameters \emph{defines an equivalence relation in $\mathcal{K}$} if, for every $\mathfrak{S}\in\mathcal{K}$, $\sim_{\phi,\psi}=\{(\vec a,\vec b)\in |\mathfrak{S}|^{2n}\mid\mathfrak{S}\vDash\phi(\vec a,\vec b)\}$ is an equivalence relation on $\{\vec a\mid\mathfrak{S}\vDash\psi(\vec a)\}$.  We call $\sim_{\phi,\psi}$ a \emph{definable equivalence relation of length $n$} and we say $\phi$ \emph{defines} $\sim_{\phi,\psi}$ \emph{inside} $\psi$.
\end{definition}

We often omit $\psi$ and simply write $\sim_\phi$ for a definable equivalence relation.

Since $\phi$ is already a formula, it does not change the automorphism group of an $\mathcal{L}$-structure if we add a predicate symbol for $\phi$ to the signature $\mathcal{L}$.  We therefore, without loss of generality, restrict our consideration to the following case.

\begin{definition}
  We say a definable equivalence relation $\sim_{\phi,\psi}$ (in a Fra\"iss\'e class $\mathcal{K}$ of $\mathcal{L}$-structures) is \emph{defined by symbols} if there are predicate symbols $P,V\in\mathcal{L}$ such that $\sim_{\phi,\psi}$ is equal to $\sim_{P\vec x\vec y,V\vec x}$.
\end{definition}

Some equivalence relations do not present an obstacle to amalgamation.  For example, the formula $\phi(x,y,z,w)$ given by $x=z$ defines an equivalence relation of length $2$.  Similarly, for any formula $\psi(x)$, the formula $\phi(x,y)$ given by $\psi(x)\leftrightarrow\psi(y)$ defines an equivalence relation of length $1$.  Both of these definable equivalence relations are present in any class of structures, including those with $n$-DAP.

More generally, if $\sim_\phi$ is some definable equivalence relation---possibly one which does prevent $n$-DAP---there are other equivalence relations refining it which do not represent ``new'' obstacles to $n$-DAP.   For instance, if $\phi(x,y)$ defines an equivalence relation of length $1$, $\phi'(x,y)$ given by $\phi(x,y)\wedge(\psi(x)\leftrightarrow\psi(y))$ defines a new equivalence relation that refines $\sim_\phi$ but is not a new obstacle to $n$-DAP.

We capture this in the notion of an explicit formula: essentially, if an equivalence relation is defined by a formula explicit in $\phi_1,\ldots,\phi_d$ then it does not represent a new obstacle to $n$-DAP beyond that represented by $\sim_{\phi_1},\ldots,\sim_{\phi_d}$.

\begin{definition}
Let $\mathfrak{M}$ be a structure and $\phi_1,\ldots,\phi_d$ and $\psi_1,\ldots,\psi_d$ be sequences of formulas where each $\phi_r$ defines an equivalence relation of length $k_r$ inside $\psi_r$.  We say a formula $\psi(x_1,\ldots,x_n,y_1,\ldots,y_m)$ is \emph{basic explicit} in $\phi_1,\ldots,\phi_d$ if it has one of the forms:
  \begin{itemize}
  \item $\psi_r(x_{i_1},\ldots,x_{i_{k_r}})\wedge\psi_r(y_{j_1},\ldots,y_{j_{k_r}})\wedge \phi_s(x_{i_1},\ldots,x_{i_{k_r}},y_{j_1},\ldots,y_{j_{k_r}})$ for some $r\leq d$ and some $i_1,\ldots,i_{k_r}\leq n$ and $j_1,\ldots,j_{k_r}\leq m$,
  \item $x_i=y_j$ for $i\leq n$, $j\leq m$,
  \item $\theta(x_1,\ldots,x_n)$ for some $\theta$, or
  \item $\theta(y_1,\ldots,y_m)$ for some $\theta$.
  \end{itemize}
A formula is \emph{explicit} if it is a boolean combination of basic explicit formulas.
A definable equivalence relation $\sim$ is \emph{explicit in $\mathfrak{M},\phi_1,\ldots,\phi_d,\psi_1,\ldots,\psi_d$} if $\sim$ is defined by an explicit formula.
\end{definition}
Explicit formulas are allowed to consider $\vec x$ and $\vec y$ \emph{separately}, but may only compare $\vec x$ with $\vec y$ using the established equivalence relations (including $=$).  Note that the common refinement of two explicit equivalence relations is also explicit.

\subsection{Strongly Orderly Sequences of Equivalence Relations}
We would like to name equivalence classes simply by assigning a number to each class, which requires knowing how many classes there are inside each class of a coarser equivalence relation.

We avoid further notational complications by ruling out the following three cases.
\begin{example}
  Let $\mathcal{L}$ have have two binary relation symbols $R$ and $S$ and a unary relation symbol $P$ and let $\mathcal{K}$ contain all structures $\mathfrak{S}$ in which $R$ and $S$ are interpreted as equivalence relations such that
  \begin{itemize}
  \item each equivalence class of $R^{\mathfrak{S}}$ and $S^{\mathfrak{S}}$ is either contained in or disjoint from $P^{\mathfrak{S}}$,
  \item when restricted to $P^{\mathfrak{S}}$, the equivalence classes of $S^{\mathfrak{S}}$ refine the classes of $R^{\mathfrak{S}}$, and
  \item when restricted to the complement of $P^{\mathfrak{S}}$, the equivalence classes of $R^{\mathfrak{S}}$ refine the classes of $S^{\mathfrak{S}}$.
  \end{itemize}

In this case, numbering equivalence classes is complicated because on $P^{\mathfrak{S}}$ we need to number a class of $S^{\mathfrak{S}}$ relative to a class of $R^{\mathfrak{S}}$, while on the complement we need to number a class of $R^{\mathfrak{S}}$ relative to a class of $S^{\mathfrak{S}}$.  We will restrict ourselves to the case where there is a single finest (in a certain sense) coarsening of each $\sim_i$, and where, further, each class of the coarser relation is refined into the same number of classes of $\sim_i$.  The example above must be handled in our framework using four equivalence relations, two on the domain $P$ and two on the domain the complement of $P$.
\end{example}
\begin{example}
Let $\mathcal{L}$ have a single quaternary relation $R$ and $\mathcal{K}$ contain those structures in which $R$ is interpreted as an equivalence relation on ordered pairs of distinct elements so that there are at most three classes and $(i,j)\not\sim_R (j,i)$.

In this case, numbering equivalence classes is complicated because we need to coordinate the assignment of numbers to $(i,j)$ and $(j,i)$.
\end{example}
\begin{example}
Let $\mathcal{L}$ have two binary relations $R$ and $S$ and let $\mathcal{K}$ contain those structures $\mathfrak{S}$ in which $R^{\mathfrak{S}}$ is an equivalence relation with at most three equivalence classes arranged in some order---say we call the classes $r_1,r_2,r_3$---and $S^{\mathfrak{S}}$ consists of those $(a,b)$ such that the label of the equivalence class in $R^{\mathfrak{S}}$ containing $b$ is one more than the label of the equivalence class in $R^{\mathfrak{S}}$ containing $a$, mod $3$, that is, $S^{\mathfrak{M}}=(r_1\times r_2)\cup (r_2\times r_3)\cup (r_3\times r_1)$.  (Note that the classes are not really labeled---in particular, in the Fra\"iss\'e limit, the classes are isomorphic.)

In this case, assigning equivalence classes is complicated because the equivalence classes are not fully symmetric.  If we attempt to permute class $r_1$ to class $r_2$, we are forced to move class $r_2$ to $r_3$.  Our framework could be extended by weakening the symmetry requirements we put on our representation, but we do not treat this further complication here.
\end{example}

We avoid these above cases by placing the following technical restrictions on the equivalence classes we consider.  It is much simpler to define these conditions using the behavior of equivalence classes in the Fra\"iss\'e limit.
\begin{definition}\label{defn:strongly orderly}

Let $\mathcal{K}$ be a Fra\"iss\'e class and $\mathfrak{M}$ its Fra\"iss\'e limit.

If $\sim,\sim^*$ are two equivalence relations of the same length, we say $\sim^*$ \emph{evenly contains} $\sim$ if $\sim^*$ is coarser than $\sim$, that is, every equivalence class of $\sim$ is contained in an equivalence class of $\sim^*$, and every equivalence class of $\sim^*$ contains the same number (from $\mathbb{N}\cup\{\infty\}$) of equivalence classes of $\sim$.

We say $\sim^*$ \emph{freely contains} $\sim$ if $\sim^*$ evenly contains $\sim$ and further, for any equivalence class $D$ of $\sim^*$, any finite set $D_1,\ldots,D_v$ of distinct equivalence classes of $\sim^*$ contained in $D$, and any permutation $\pi:[v]\rightarrow[v]$, there is an automorphism $\hat \pi$ of $\mathfrak{M}$ mapping $D_i$ to $D_{\pi(i)}$ and mapping all equivalence classes of $\sim$ not in $D$ to themselves.

If $\sim,\sim',\sim^*$ are three equivalence relations of the same length and $\sim'$ refines $\sim^*$, we say $\sim$ is \emph{orthogonal to} $\sim'$ within $\sim^*$ if $D\cap D'\neq\emptyset$ whenever $D^*$ is an equivalence class of $\sim^*$ and $D,D'$ are equivalence classes of $\sim,\sim'$, respectively, with $D,D'\subseteq D^*$.

When $\sim$ evenly contains $\sim'$, we write $\#_{\sim}(\sim')$ for the number (in $\mathbb{N}\cup\{\infty\})$ of classes of $\sim'$ in each class of $\sim$ (in $\mathfrak{M}$).

  We say a sequence of definable equivalence relations $\sim_{1},\ldots,\sim_{d}$ in $\mathcal{K}$ is \emph{orderly} if for each $r\leq d$ there is a definable equivalence relation $\sim_r^*$ such that
  \begin{itemize}
  \item $\sim_r^*$ is explicit in $\sim_1,\ldots,\sim_{r-1}$,
  \item $\sim_r^*$ freely contains $\sim_r$ (in $\mathfrak{M}$),
  \item if $\#_{\sim_r^*}(\sim_r)=\infty$ then $\sim_r$ has length $1$, and
  \item every equivalence relation explicit in $\sim_1,\ldots,\sim_{r-1}$ and finer than $\sim_r^*$ is orthogonal to $\sim_r$ (in $\mathfrak{M}$).
  \end{itemize}

We write $\#(\sim_r)$ for $\#_{\sim_r^*}(\sim_r)$.  When $\vec x=\langle x_1,\ldots,x_k\rangle$ is a sequence, we write $\rng\vec x$ for the set $\{x_1,\ldots,x_r\}$.

We say $\sim_1,\ldots,\sim_d$ is \emph{strongly orderly} if $\sim_1,\ldots,\sim_d$ is orderly and if, for every $r$ and every $\vec x,\vec y$ where $|\vec x|=|\vec y|$ is the length of $\sim_r$ and $\rng\vec x=\rng\vec y$, $\vec x\sim_r^*\vec y$ implies $\vec x\sim_r\vec y$.
\end{definition}

\subsection{Naming Equivalence Classes}

\begin{definition}
  Let $\mathcal{L}=\{R_1,\ldots,R_k\}$ be a signature, let $\mathcal{K}$ be a Fra\"iss\'e class of $\mathcal{L}$-structures and suppose $\sim_1,\ldots,\sim_d$ is a strongly orderly sequence of definable equivalence relations in $\mathcal{K}$ so that $\sim_r$ has length $k_r$ for each $r=1,\ldots,d$.
For $\mathfrak{S}\in\mathcal{K}$, we write $V_{r,\mathfrak{S}}$ as the domain of $\sim_r$ in $\mathfrak{S}$ and
\[U_{r,\mathfrak{S}}=V_{r,\mathfrak{S}}/\sim_r.\]

A \emph{partition labeling} of $\mathfrak{S}$ is a family of injective functions $(\eta^r)_{1\leq r\leq d}$ with each $\eta^r:U_{r,\mathfrak{S}}\rightarrow [\#(\sim_r)]$.
\end{definition}
By abuse of notation, we often interpret $\eta^r$ as a function on $V_{r,\mathfrak{S}}$ by composing with the projection onto $U_{r,\mathfrak{S}}$, writing $\eta^r(x)$ in place of $\eta^r([x]_{\sim_r})$ as convenient.  Importantly, a partition labeling assigns $x$ and $y$ the same value if and only if they belong to the same equivalence class.

When $\#(\sim_r)=\infty$, we often relax the requirement that the range of $\eta^r$ be $\mathbb{N}$, since it does no harm to replace the range with some other countable set.  When $\#(\sim_r)\in\mathbb{N}$, we always insist that the range of $\eta^r$ be exactly $[\#(\sim_r)]$ because we need to ensure that the range is the correct size.

\begin{definition}\label{defn:upto}
Suppose that $(\mathfrak{S}_j)_{j\leq n}$ is an amalgamation plan and that $\eta^r_j$ is a partition labeling of $\mathfrak{S}_j$ for each $1\leq j\leq n$.  We call $(\eta^r_j)$ \emph{coherent} if $\eta^r_j(\vec y)=\eta^r_{j'}(\vec y)$ whenever $\rng\vec y= s\subseteq|\mathfrak{S}_j|\cap|\mathfrak{S}_{j'}|$ and $\vec y\in V_{r,\mathfrak{S}_j}$ (so also $\vec y\in V_{\mathfrak{S}_{j'}}$).

We say $\mathcal{K}$ \emph{has $n$-DAP up to $\sim_1,\ldots,\sim_d$} if, whenever $(\mathfrak{S}_j)_{j\leq n}$ is an amalgamation plan with each $\mathfrak{S}_j\in\mathcal{K}$ and $(\eta^r_j)_{r,j}$ is a coherent sequence of partition labelings, there is an amalgam $\mathfrak{S}\in\mathcal{K}$.
\end{definition}
The $\eta^r_j$ can be viewed as labels to equivalence classes.  Then $n$-DAP up to $\sim_1,\ldots,\sim_d$ says that we can amalgamate an amalgamation plan as long as the equivalence classes in the $\mathfrak{S}_j$ can be reconciled in a coherent way.

We note, for later use, that once we have $\odap$-DAP, we can amalgamate more complicated collections of structures as well.
\begin{theorem}\label{thm:complex_amalgamation}
  Suppose $\mathcal{K}$ has $n$-DAP up to $\sim_1,\ldots,\sim_d$, that $(\mathfrak{S}_i)_{i\leq n}\subseteq\mathcal{K}$ for each $i\leq n$, that $(\eta^r_i)$ is a coherent partition labeling of the $\mathfrak{S}_i$, and that $\mathfrak{S}_i|_S$ is identical to $\mathfrak{S}_{i'}|_S$ whenever $S\subseteq|\mathfrak{S}_i|\cap|\mathfrak{S}_{i'}|$.  Then there exists $\mathfrak{S}\in\mathcal{K}$ such that $|\mathfrak{S}|=\bigcup_{i\leq n}|\mathfrak{S}_i|$ and $\mathfrak{S}|_{|\mathfrak{S}_i|}$ is identical to $\mathfrak{S}_i$ for every $i\leq n$.
\end{theorem}
\begin{proof}
We proceed by induction on $\bigcup_{i\leq n}|\mathfrak{S}_i|$.  Consider $S_0=\bigcup_{i\leq n}|\mathfrak{S}_i|\setminus\{x\}$ for some $x$, so that $\#S_0=\#(\bigcup_{i\leq n}|\mathfrak{S}_i|)-1$, where $\#A$ denotes the cardinality to avoid ambiguity, and consider all the structures $\mathfrak{S}'_i=\mathfrak{S}_i|_{|\mathfrak{S}_i|\cap S_0}$ with the restriction of the partition labeling to this set.  By the inductive hypothesis, there is an amalgam $\mathfrak{S}_x$, and we may choose a partition labeling $(\eta^r_x)$ coherent with the $(\eta^r_i)$.

Repeating this process for each choice of $x$, we have an amalgamation plan and a coherent partition labeling.  We apply the $\odap$-DAP of $\mathcal{K}$ up to $\sim_1,\ldots,\sim_d$ to obtain the desired $\mathfrak{S}$.
\end{proof}

\section{Eliminating Equivalence Relations}\label{sec:eliminating}

In this section we take a Fra\"iss\'e class $\mathcal{K}$ with $2$-DAP and with $\odap$-DAP up to some strongly orderly sequence $\sim_1,\ldots,\sim_d$ in a signature $\mathcal{L}$ and we expand both the signature and the structures to obtain a related class with $\odap$-DAP.  We always assume that $\sim_r$ is a definable equivalence relation of length $k_r$ and is defined by symbols.

\subsection{Finitely Many Classes}

Our approach to finite equivalence relations is straightforward: we add finitely many new predicate symbols naming each equivalence class.  The only complication is that we need to make precise, and then keep track of, the way in which these new predicate symbols are symmetric.

Throughout this subsection, we assume $\#(\sim_d)=v$ is finite.

\begin{definition}
Define a signature $\mathcal{L}_{\sim_d,v}$ extending $\mathcal{L}$ by $v$ new $k_d$-ary relations $P_{\sim_d,1},\ldots,P_{\sim_d,v}$.  

When $\mathfrak{S}$ is an $\mathcal{L}_{\sim_d,v}$-structure, write $\mathfrak{S}^-$ for the $\mathcal{L}$-structure obtained by forgetting the interpretations of the symbols $P_{\sim_d,i}$, so that $\mathfrak{S}^-$ is the \emph{reduct} of $\mathfrak{S}$ to $\mathcal{L}$.  When $\mathfrak{S}$ is an $\mathcal{L}$-structure and $\eta_d$ is a partition labeling, we write $\mathfrak{S}^{\eta_d}$ for the $\mathcal{L}_{\sim_d,v}$-structure obtained by extending $\mathfrak{S}$ with $P_{\sim_d,i}^{\mathfrak{S}^{\eta_d}}=\{\vec x\mid \eta_d(\vec x)=i\}$.

We define $\mathcal{K}_{\sim_d}$ to consist of those finite $\mathcal{L}_{\sim_d,v}$-structures of the form $\mathfrak{S}^{\eta_d}$ for some $\mathfrak{S}\in\mathcal{K}$ and $\eta_d$ a partition labeling.
\end{definition}

\begin{lemma}
  $\mathcal{K}_{\sim_d}$ is a Fra\"iss\'e class with $2$-DAP.
\end{lemma}

We need to characterize the way in which the relations $P_{\sim_d,i}$ are symmetric.

\begin{definition}
  If $P_1,\ldots,P_v$ are $k$-ary relations in $\mathcal{L}$, $\mathfrak{S}$ is an $\mathcal{L}$-structure, $\pi:[v]\rightarrow[v]$ is a permutation, and $S\subseteq|\mathfrak{S}|^{k}$, we define $\mathfrak{S}_{\pi\mid S}$ to be the structure given by
  \begin{itemize}
  \item $|\mathfrak{S}_{\pi\mid S}|=|\mathfrak{S}|$,
  \item $P_d^{\mathfrak{S}_{\pi\mid S}}=(P_{\pi^{-1}(d)}^{\mathfrak{S}}\cap S)\cup(P_d^{\mathfrak{S}}\setminus S)$,
  \item $R^{\mathfrak{S}_{\pi\mid S}}=R^{\mathfrak{S}}$ for any other relation symbol.
  \end{itemize}

When $\sim^*$ is a $k$-ary equivalence relation, we say a Fra\"iss\'e class $\mathcal{K}$ is \emph{symmetric in} $\{P_1,\ldots,P_v\}$ \emph{within} $\sim^*,V$ if $\mathfrak{S}_{\pi\mid S\cap V^{\mathfrak{S}}}\in\mathcal{K}$ whenever $\mathfrak{S}\in\mathcal{K}$, $S$ is a $\sim^*$-equivalence class, and $\pi$ is a permutation of $[v]$.
\end{definition}

It is immediate from the definition and the fact that $\sim^*_d$ freely contains $\sim_d$ that $\mathcal{K}_{\sim_d}$ is symmetric in $\{P_{\sim_d,1},\ldots,P_{\sim_d,v}\}$ within $\sim^*_d,V_{d,\mathcal{K}_{\sim_d}}$.  (Recall that $V_{d,\mathcal{K}_{\sim_d}}$ is the predicate defining the domain of $\sim_d$.)

Furthermore---because we will need to apply this process iteratively---we note that this definition preserves existing symmetries.
\begin{lemma}
  Suppose that $\mathcal{K}$ is symmetric in $\{Q_1,\ldots,Q_w\}$ within $\sim^*,V$ and that $\sim_d$ has a definition not involving any predicate $Q_i$.  Then $\mathcal{K}_{\sim_d}$ is also symmetric in $\{Q_1,\ldots,Q_w\}$ within $\sim^*,V$.
\end{lemma}
\begin{proof}
  Consider some $\mathfrak{S}\in\mathcal{K}$, some equivalence class $S$ of $\sim^*$, and some permutation $\pi$ of $[w]$.  Let $\mathfrak{S}'$ be the reduct of $\mathfrak{S}$ to $\mathcal{L}$.  Then $\mathfrak{S}'\in\mathcal{K}$, so also $\mathfrak{S}'_{\pi\mid S\cap V}\in\mathcal{K}$.  $\mathfrak{S}$ induces a partition labeling $\eta$ on $\mathfrak{S}'$, and $\eta$ is also a partition labeling on $\mathfrak{S}'_{\pi\mid S\cap V}$, so $\mathfrak{S}'_{\pi\mid S\cap V}$ and $\eta$ induce an element $\mathfrak{S}_{\pi\mid S\cap V}\in\mathcal{K}$.
\end{proof}

\begin{lemma}
  $\mathcal{K}_{\sim_d}$ has $\odap$-DAP up to the strongly orderly sequence $\sim_1,\ldots,\sim_{d-1}$ and the equivalence relations $\sim_r$ for $r\leq d-1$ have definitions not involving any of the predicates $P_{\sim_d,i}$.
\end{lemma}
\begin{proof}
It is easy to see that $\sim_1,\ldots,\sim_{d-1}$ remain a strongly orderly sequence defined by the same formulas as in $\mathcal{K}$.

  Let $\{\mathfrak{S}_j\}_{j\leq n}$ be an amalgamation plan with each $\mathfrak{S}_j\in\mathcal{K}_{\sim_d}$ and $(\eta^r_j)_{r<d,j}$ a coherent sequence of partition labelings.  Define structures $\mathfrak{S}'_j\in\mathcal{K}$ to be the $\mathcal{L}$-reduct of $\mathfrak{S}_j$, and define a partition labeling $\eta^d_j$ on $\mathfrak{S}'_j$ by $\eta^d_j(\vec x)=i$ if and only if $\vec x\in{P}_{\sim_d,i}^{\mathfrak{S}_j}$.  Then $\{\mathfrak{S}'_j\}_{j\leq n}$ is an amalgamation plan and $(\eta^r_j)_{r\leq d,j}$ is a coherent sequence of partition labelings, so there is an amalgam $\mathfrak{S}'\in\mathcal{K}$.  Taking any partition labeling $\eta^r$ on $\mathfrak{S}'$ extending $\bigcup_j\eta^r$, the pair $\mathfrak{S}',\eta^r$ induces a $\mathcal{L}_{\sim_d}$-structure $\mathfrak{S}$ by setting $P_{\sim_d,i}^{\mathfrak{S}}=\{\vec x\mid \eta^r(\vec x)=i\}$.  Then $\mathfrak{S}$ is the amalgam of $\{\mathfrak{S}_j\}_{j\leq n}$.
\end{proof}

Observe that if $\mathfrak{X}$ is $\mathcal{K}$-exchangeable, then it induces a natural $\mathcal{K}_{\sim_d}$-exchangeable structure $\mathfrak{X}_{\sim_d}$ by putting $\mathfrak{X}_{\sim_d}(\mathfrak{S})=\mathfrak{X}(\mathfrak{S}^-)$.

\begin{lemma}\label{thm:lift_fin}
  Suppose $\mathfrak{X}$ is a $\mathcal{K}$-exchangeable $\mathcal{L}'$-structure, with $\mathfrak{X}_{\sim_d}$ as defined above.  Then $\mathfrak{X}_{\sim_d}(\mathfrak{S})\equalinlaw\mathfrak{X}_{\sim_d}(\mathfrak{S}_{\pi|S})$  for any $\mathfrak{S}\in\mathcal{K}_{\sim_d}$, any $\pi:[v]\rightarrow[v]$, and any equivalence class $S$ of $\sim^*_d$.
\end{lemma}

\subsection{Infinitely Many Classes}

The case where $\#(\sim_d)$ is infinite must be handled differently than the case of finitely many classes.  We add new elements which represent the equivalence classes of $\sim_d$, and a unary predicate symbol $C^{\sim_d}$ which we interpret as the part of the structure naming equivalence classes.

Throughout this subsection, we assume that $\#(\sim_d)=\infty$, which implies $k_d=1$ by our convention.  We let $V$ be the unary symbol defining the domain of $\sim_d$.

Our plan is to define a new signature $\mathcal{L}_{\sim_d}$ and consider models $\mathfrak{S}$ in $\mathcal{L}_{\sim_d}$ whose domain is partitioned into three sets, $C^{\mathfrak{S}}$, $V^{\mathfrak{S}}$, $E^{\mathfrak{S}}$.  The corresponding $\mathcal{L}$-structure $\mathfrak{S}^-$ is given by \emph{pairs} of elements from $\mathfrak{S}$, which must either be of the form $(v,c)\in V^{\mathfrak{S}}\times C^{\mathfrak{S}}$ or $(e,e)$ for some $e\in E^{\mathfrak{S}}$.  The pairs $(v,c)$ represent elements of $V^{\mathfrak{S}^-}$ where we have disentangled the element $v$ from its equivalence class $c$.  The pairs $(e,e)$ represent elements of $E^{\mathfrak{S}}$ (doubled to preserve arities).

A typical predicate $Px_1\cdots x_n$ from $\mathcal{L}$ becomes a predicate $Px_1y_1\cdots x_ny_n$ in $\mathcal{L}_{\sim_d}$.  However, unary predicates have to remain unary (to preserve the property that $k_r=1$ for $r<d$ if $\#(\sim_r)=\infty$) and similarly some binary predicates (those defining $\sim_r$ with $k_r=1$) also have to remain binary.  To allow this, we have to decide, for each predicate $P$, whether it should be a property of $v$ or of $c$.

\begin{definition}
We define a signature $\mathcal{L}_{\sim_d}$ as follows:
    \begin{itemize}
    \item $\mathcal{L}_{\sim_d}$ contains a fresh unary predicate $C$ and
    \item for each $m$-ary relation symbol $P$ of $\mathcal{L}$
      \begin{itemize}
      \item if $P$ defines an equivalence relation $\sim_r$ with $k_r=1$ then $\mathcal{L}_{\sim_d}$ contains an binary relation symbol also denoted $P$,
      \item if $m=1$ then $\mathcal{L}_{\sim_d}$ contains a unary relation symbol also denoted $P$, and
      \item otherwise $\mathcal{L}_{\sim_d}$ contains a $2m$-ary relation symbol also denoted $P$.
      \end{itemize}
    \end{itemize}
\end{definition}

  If $P$ is a unary relation, we let $\sim_P$ be the equivalence relation $x\sim_P y$ if and only if $(Px\leftrightarrow Py)\wedge x\sim_d^*y$.  Since $\sim_P$ is clearly explicit in $\sim_1,\ldots,\sim_{d-1}$ and at least as fine as $\sim^*_d$, either $\sim_P$ is equal to $\sim_d^*$ or $\sim_P$ is orthogonal to $\sim_d$.

Similarly, if $P$ is a binary relation defining an equivalence relation $\sim_r$, $r<d$, with $k_r=1$ then either $\sim_r\cap\sim^*_d$ is equal to $\sim_d^*$ or orthogonal to $\sim_d$.  (Note that, by the usual set-theoretic definition of an equivalence relation, $\sim_r\cap\sim^*_d$ is precisely the common refinement of these equivalence relations.)

\begin{definition}\label{defn:meaningful}
  When $\mathfrak{S}$ is a $\mathcal{L}_{\sim_d}$-structure, we write $E^{\mathfrak{S}}=|\mathfrak{S}|\setminus(C^{\mathfrak{S}}\cup V^{\mathfrak{S}})$.

  We write $\mathcal{U}_n(\mathfrak{S})$ for those $2n$-tuples $\langle x_1,y_1,\ldots,x_n,y_n\rangle$ such that for each $i\leq n$ either $x_i\in V^{\mathfrak{S}}$ and $y_i\in C^{\mathfrak{S}}$ or $x_i=y_i\in E^{\mathfrak{S}}$.

  We say an $\mathcal{L}_{\sim_d}$-structure $\mathfrak{S}$ is \emph{meaningful} if:
  \begin{itemize}
  \item $C^{\mathfrak{S}}\cap V^{\mathfrak{S}}=\emptyset$,
  \item if $P$ is unary and $\sim_P$ is equal to $\sim_d^*$ then $P^{\mathfrak{S}}\cap V^{\mathfrak{S}}=\emptyset$,
  \item if $P$ is unary and $\sim_P$ is orthogonal to $\sim_d$ then $P^{\mathfrak{S}}\cap C^{\mathfrak{S}}=\emptyset$,
  \item if $P$ is binary, defines $\sim_r$, and $\sim_r\cap\sim^*_d$ is equal to $\sim_d^*$ then $P^{\mathfrak{S}}\subseteq (C^{\mathfrak{S}}\cup E^{\mathfrak{S}})\cup(C^{\mathfrak{S}}\cup E^{\mathfrak{S}})$,
  \item if $P$ is binary, defines $\sim_r$, and $\sim_r\cap\sim^*_d$ is orthogonal to $\sim_d$ then $P^{\mathfrak{S}}\subseteq (V^{\mathfrak{S}}\cup E^{\mathfrak{S}})\times (V^{\mathfrak{S}}\cup E^{\mathfrak{S}})$, and
  \item all other relation symbols have $P^{\mathfrak{S}}\subseteq \mathcal{U}_n(\mathfrak{S})$.
  \end{itemize}

  If $\mathfrak{S}$ is a meaningful $\mathcal{L}_{\sim_d}$-structure then we define $\mathfrak{S}^-$ by setting $|\mathfrak{S}^-|=(V^{\mathfrak{S}}\times C^{\mathfrak{S}})\cup E^{\mathfrak{S}}$.  We define projections
  \begin{itemize}
  \item $\pi_V:\mathfrak{S}^-\rightarrow\mathfrak{S}$ by $\pi_V(v,c)=v$ and $\pi_V(e)=e$,
  \item $\pi_C:\mathfrak{S}^-\rightarrow\mathfrak{S}$ by $\pi_C(v,c)=c$ and $\pi_C(e)=e$, and
  \item $\pi_U:\mathfrak{S}^-\rightarrow\mathcal{U}_1(\mathfrak{S})$ by $\pi_U(v,c)=(v,c)$ and $\pi_U(e)=(e,e)$.
  \end{itemize}
We then define relation symbols as follows:
  \begin{itemize}
  \item if $P$ is unary and $\sim_P$ is equal to $\sim_d^*$ then $P^{\mathfrak{S}^-}=\pi_C^{-1}(P^{\mathfrak{S}})$,
  \item if $P$ is unary and $\sim_P$ is orthogonal to $\sim_d$ then $P^{\mathfrak{S}^-}=\pi_V^{-1}(P^{\mathfrak{S}})$,
  \item if $P$ is binary, defines $\sim_r$, and $\sim_r\cap\sim^*_d$ is equal to $\sim_d^*$ then $P^{\mathfrak{S}^-}=\{(x,y)\mid (\pi_C(x),\pi_C(y))\in P^{\mathfrak{S}}\}$,
  \item if $P$ is binary, defines $\sim_r$, and $\sim_r\cap\sim^*_d$ is orthogonal to $\sim_d^*$ then $P^{\mathfrak{S}^-}=\{(x,y)\mid (\pi_V(x),\pi_V(y))\in P^{\mathfrak{S}}\}$, and
  \item otherwise $P^{\mathfrak{S}^-}=\{(x_1,\ldots,x_n)\mid (\pi_U(x_1),\ldots,\pi_U(x_n))\in P^{\mathfrak{S}}\}$.
  \end{itemize}
\end{definition}

Though tempting to define $\mathcal{K}_{\sim_d}$ to consist of precisely those meaningful $\mathfrak{S}$ such that $\mathfrak{S}^-\in\mathcal{K}$, this choice gives the wrong outcome if one of $V^{\mathfrak{S}}$ and $C^{\mathfrak{S}}$ were empty and the other were not.  Furthermore, since $\mathcal{K}_{\sim_d}$ must be closed under substructures, we must determine which such structures should be included.

\begin{definition}
  We say $\mathfrak{S}$ is \emph{large enough} if $|C^{\mathfrak{S}}|\geq 1$ and $|V^{\mathfrak{S}}|\geq 1$.  We define $\mathcal{K}_{\sim_d}$ consists of those $\mathfrak{S}$ such that there is some meaningful, large enough $\mathfrak{S}_0\supseteq\mathfrak{S}$ so that $\mathfrak{S}_0^-\in\mathcal{K}$.
\end{definition}

The definition of $\mathfrak{S}^-$ suggests why we include the requirement that $k_d=1$ whenever $\#(\sim_d)=\infty$, for otherwise the quotienting process would be much more complicated.
If, say, $k_d=2$, then we ought to turn pairs $(x,y)$ in $\mathcal{K}$ into triples $(x,y,z)$ in $\mathcal{K}(\sim_d)$, with each pair $(x,y)$ corresponding to many similar but non-equivalent from $\mathcal{K}$ with different equivalence classes; but then it is not clear how to assign meaning to the individual point $x$.

\begin{lemma}
  $\mathcal{K}_{\sim_d}$ is a Fra\"iss\'e class with $2$-DAP.
\end{lemma}

Our definition ensures that the equivalence relations carry over essentially unchanged.
\begin{lemma}
  $\sim_1,\ldots,\sim_{d-1}$ is a strongly orderly sequence in $\mathcal{K}_{\sim_d}$.
\end{lemma}

\begin{lemma}
$\mathcal{K}_{\sim_d}$ has $\odap$-DAP up to $\sim_1,\ldots,\sim_{d-1}$.
\end{lemma}
\begin{proof}
It is convenient to directly prove the stronger version of amalgamation from Theorem \ref{thm:complex_amalgamation}.  Let $\{\mathfrak{S}_i\}_{i\leq n}\subseteq\mathcal{K}_{\sim_d}$ and a coherent partition labeling $(\eta^r_i)_{r<d}$ be given so that $\mathfrak{S}_i|_S$ is identical to $\mathfrak{S}_{i'}|_S$ whenever $S\subseteq|\mathfrak{S}_i|\cap|\mathfrak{S}_{i'}|$.  Enlarging models if necessary, assume that each $\mathfrak{S}_i$ is large enough. (These expansions need not overlap in the way required by the statement of $n$-DAP, which is why we prove the stronger version.)

Each $\mathfrak{S}_i$ has a corresponding $\mathfrak{S}_i^-\in\mathcal{K}$.  Observe that $S\subseteq ((V^{\mathfrak{S}_i}\times C^{\mathfrak{S}_i})\cup E^{\mathfrak{S}_i})\cap ((V^{\mathfrak{S}_{i'}}\times C^{\mathfrak{S}_{i'}})\cup E^{\mathfrak{S}_{i'}})$ whenever $S\subseteq|\mathfrak{S}^-_i|\cap|\mathfrak{S}^-_{i'}|$ and, therefore, $\mathfrak{S}^-_i|_S$ is identical to $\mathfrak{S}^-_{i'}|_S$.  We can also define a corresponding partition labeling $\eta^{r,-}_i$: for $r<d$, this is clear from the definition of $\mathfrak{S}^-$, and we set $\eta^d_i(v,c)=c$ for each $\mathfrak{S}^-_i$.

Amalgamating these structures may not be enough, however, because there might be pairs $(v,c)$ where $v\in V^{\mathfrak{S}_i}$ and $c\in (C)^{\mathfrak{S}_{i'}}$, but the corresponding pair $(v,c)$ is not contained in $\bigcup_i|\mathfrak{S}_i^-|$.  We add such elements manually: for each such pair we include a one point structure $\mathfrak{S}^-_{v,c}$ with $|\mathfrak{S}_{v,c}^-|=\{(v,c)\}$ and relations:
\begin{itemize}
\item if $P$ is unary and $\sim_P$ is $\sim^*_d$ then $(v,c)\in P^{\mathfrak{S}^-_{v,c}}$ if and only if $c\in P^{\mathfrak{S}_{i'}}$, and
\item if $P$ is unary and $\sim_P$ is orthogonal to $\sim_d$ then $(v,c)\in P^{\mathfrak{S}^-_{v,c}}$ if and only if $v\in P^{\mathfrak{S}_i}$.
\end{itemize}
Observe that some structure with these properties exists: consider the equivalence relation $\sim_{e}$ definable in $\mathcal{K}$ so that $x\sim_e y$ if $x\sim^*_d y$ and for each $P$ such that $\sim_P$ is orthogonal to $\sim_d$, $x\in P\leftrightarrow y\in P$.  This equivalence relation must be orthogonal to $\sim_d$, and so we choose any element $(w,c)\in P^{\mathfrak{S}^-_{i'}}$ along with an extension $\mathfrak{S}^{*}_{i'}$ so that there is some element in $\mathfrak{S}^{*}_{i'}$ in the same $\sim_d$ equivalence class as $(v,c)$ and in the same $\sim_e$ equivalence class as $(v,c)$.  Then we may take $\mathfrak{S}_{v,c}^-=\mathfrak{S}^{*}_{i'}|_{\{(w,c)\}}$.  We then set $\eta^d_{v,c}(v,c)=c$ and choose the other $\eta^r_{v,c}(v,c)$ equal to $\eta^r_i(v,c')$ or $\eta^r_{i'}(v',c)$ respectively, depending on whether $\sim_r\cap\sim^*_d$ is equal to $\sim^*_d$.

By Theorem \ref{thm:complex_amalgamation}, we can amalgamate these structures to get $\mathfrak{S}^-$, which induces an amalgam $\mathfrak{S}$.
\end{proof}

\begin{lemma}\label{thm:inf_embedding}
  Let $\mathfrak{S}$ be an $\mathcal{L}$-structure.  Then there is an $\mathcal{L}_{\sim_d}$-structure $\mathfrak{S}_{\sim_d}$ such that $\mathfrak{S}$ is isomorphic to a substructure of $\mathfrak{S}^-_{\sim_d}$.
\end{lemma}
\begin{proof}
Without loss of generality, we may expand the structure $\mathfrak{S}$.  First, we may assume that for all $v,w\in V^{\mathfrak{S}}$ and all $\sim_r$ such that $\sim_r\cap\sim^*_d$ is orthogonal to $\sim_d$ there is a $u$ with $u\sim_r v$, $u\not\sim_d v$, and $u\not\sim_d w$.

Next, for any $v,w\in V^{\mathfrak{S}}$ such that $v\sim^*_dw$ but $v\not\sim_d w$, we wish to ensure that there is a $v'$ so that $v'\sim_d w$ which has roughly the same type as $v$.  We let $\mathfrak{S}_{-w}$ be the substructure of $\mathfrak{S}$ that does not contain any element in the same $\sim_d$ equivalence class as $w$.  We choose a structure $\mathfrak{S}_{v'}$ which contains $w$ and contains an element $v'$ which has the property that for each unary relation $P$ with $\sim_P$ orthogonal to $\sim_d$, $v'\in P^{\mathfrak{S}_{v'}}$ if and only if $v\in P^{\mathfrak{S}}$.  
We then amalgamate these structures to obtain $\mathfrak{S}'$ with a single new element $v'$ so that $v'\sim_d w$ and $v$ and $v'$ are identical as elements of $\mathfrak{S}_{-w}$.

  We take $|\mathfrak{S}_{\sim_d}|=|\mathfrak{S}|\cup U_{d,\mathfrak{S}}$.  Naturally, we set $C^{\mathfrak{S}_{\sim_d}}=U_{d,\mathfrak{S}}$, which is sufficient to define the universe of $\mathfrak{S}^-_{\sim_d}$ and the expected embedding $\pi:\mathfrak{S}\rightarrow\mathfrak{S}^-_{\sim_d}$ given by mapping $e\in |\mathfrak{S}|\setminus E^{\mathfrak{S}}$ to itself and $v\in V^{\mathfrak{S}}$ to the pair $(v,c)$ where $c\in U_{d,\mathfrak{S}}$ is the $\sim_d$ equivalence class of $v$.  In particular, note that the function $\pi_C\circ \pi$ maps an element of $\mathfrak{S}$ to its equivalence class in $\mathfrak{S}_{\sim_d}$.

For each unary relation $P$ of $\mathcal{L}$ where $\sim_P$ is equal to $\sim^*_d$, $P^{\mathfrak{S}_{\sim_d}}$ consists of those $a$ such that $(\pi_C\circ \pi)^{-1}(a)\subseteq P^{\mathfrak{S}}$.  For each unary relation $P$ of $\mathcal{L}$ where $\sim_P$ is orthogonal to $\sim_d$, we set $P^{\mathfrak{S}_{\sim_d}}=P^{\mathfrak{S}}$.  Similarly, for each binary $P$ defining a relation $\sim_r$ with $\sim_r\cap \sim^*_d$ equal to $\sim^*_d$, $P^{\mathfrak{S}_{\sim_d}}$ consist of those pairs $(a,b)$ such that $(a,b)\in P^{\mathfrak{S}}$ for some (equivalently, every) $(a',b')$ with $\pi_C\circ\pi(a')=a$ and $\pi_C\circ\pi(b')=b$.  For each binary relation $P$ defining a relation $\sim_r$ with $\sim_r\cap\sim^*_d$ orthogonal to $\sim_d$, $P^{\mathfrak{S}_{\sim_d}}=P^{\mathfrak{S}}$.  
For all other relation symbols $P$ we map a tuple $(x_1,y_1,\ldots,x_n,y_n)$ from $\mathfrak{S}_{\sim_d}$ to a corresponding tuple $(x'_1,\ldots,x'_n)$ and set $(x_1,y_1,\ldots,x_n,y_n)\in P^{\mathfrak{S}_{\sim_d}}$ if and only if $(x'_1,\ldots,x'_n)\in P^{\mathfrak{S}}$.

Our claim then follows from the definitions that $\mathfrak{S}\subseteq \mathfrak{S}^-_{\sim_d}$.
\end{proof}

\begin{definition}
  If $\mathfrak{X}$ is a $\mathcal{K}$-exchangeable $\mathcal{L}'$-structure, we define a signature $\mathcal{L}'_{\mathrm{dbl}}$ where, for each $m$-ary relation symbol of $\mathcal{L}'$, $\mathcal{L}'_{\mathrm{dbl}}$ contains a $2m$-ary  symbol.  

Given a $\mathcal{L}'$-structure $\mathfrak{T}$, we define a $\mathcal{L}'_{\mathrm{dbl}}$-structure $\mathfrak{T}^{\mathrm{dbl}}$ to be the structure where
\[(x_1,\ldots,x_{2m})\in P^{\mathfrak{T}^{\mathrm{dbl}}}\]
iff $(\pi_U^{-1}(x_1,x_2),\ldots,\pi_U^{-1}(x_{2m-1},x_{2m}))\in P^{\mathfrak{T}}$, where $\pi_U$ is defined in Definition \ref{defn:meaningful}.

We define a random $\mathcal{L}'_{\mathrm{dbl}}$-structure $\mathfrak{X}_{\mathrm{dbl}}$ by setting
\[\mathfrak{X}_{\mathrm{dbl}}(\mathfrak{S})=(\mathfrak{X}(\mathfrak{S}^-))^{\mathrm{dbl}}.\]

Conversely, given a $\mathcal{L}'_{\mathrm{dbl}}$-structure $\mathfrak{T}$ with the same universe as a $\mathcal{L}_{\mathrm{dbl}}$-structure $\mathfrak{S}$, we define a structure $\mathfrak{T}^-$ with the same universe as $\mathfrak{S}^-$ by setting $(x_1,\ldots,x_m)\in P^{\mathfrak{T}^-}$ if and only if $(\pi_U(x_1),\ldots,\pi_U(x_m))\in P^{\mathfrak{T}}$.
\end{definition}

\begin{lemma} With $\mathfrak{X}_{\mathrm{dbl}}$ and everything else as defined above, we have
  \[\mathbb{P}(\mathfrak{X}_{\mathrm{dbl}}(\mathfrak{S})=\mathfrak{T})=\mathbb{P}(\mathfrak{X}(\mathfrak{S}^-)=\mathfrak{T}^-).\]
\end{lemma}
\begin{proof}
Immediate from the observation that $(\mathfrak{T}^{\mathrm{dbl}})^-=\mathrm{T}$.
\end{proof}

\begin{lemma}\label{thm:lift_infin}
$\mathfrak{X}_{\mathrm{dbl}}$ is $\mathcal{K}_{\sim_d}$-exchangeable.
\end{lemma}

\section{Representations over Structures with Equivalence Relations}\label{sec:representations}

Recall the definitions and notation set forth in Sections \ref{sec:setting} and \ref{sec:sequences}.
\subsection{Lifting Over Equivalence Relations with Finitely Many Classes}

We show that, given a suitable representation of the desired kind over $\mathcal{K}_{\sim_d}$, we can also produce a representation over $\mathcal{K}$.

\begin{theorem}\label{thm:lift_finite}
  Let $\mathcal{K}$ be a Fra\"iss\'e class in the signature $\mathcal{L}$ with $2$-DAP and with $\odap$-DAP up to a strongly orderly sequence $\sim_1,\ldots,\sim_d$ where $\#(\sim_d)$ is finite.  Suppose that $\mathfrak{X}$ is a $\mathcal{K}$-exchangeable $\mathcal{L}'$-structure and that there are eq-symmetric functions $\{f_n\}$ so that $\mathfrak{X}^{\{f_n\}}=_{\mathcal{D}}\mathfrak{X}_{\sim_d}$.
Then there are eq-symmetric functions $\{f'_n\}$ so that $\mathfrak{X}^{\{f'_n\}}=_{\mathcal{D}}\mathfrak{X}$.
\end{theorem}
\begin{proof}
We are given functions $f_n$ so that, choosing i.i.d.\ Uniform$[0,1]$ random variables $(\xi_s)$, independent uniform random orderings $(\prec_s)$, and random labelings $(\eta_r)_{r<d}$, we set
\[\tp_{\mathfrak{X}^{\{f_n\}}}(s)=f_{|s|}(\mathfrak{S}|_s,(\xi_\tau)_{\tau\in B(s)},(\prec_\tau)_{\tau\in B(s)},(\eta_r\upharpoonright s)).\]

  We define a function $f'$ from $f$ by setting
\[f'(\mathfrak{S},(\xi_\tau)_{\tau\in B(s)},(\prec_\tau)_{\tau\in B(s)},(\eta_r)_{r\leq d})\]
equal to
\[f(\mathfrak{S}^{\eta_d},(\xi_\tau)_{\tau\in B(s)},(\prec_\tau)_{\tau\in B(s)},(\eta_r)_{r< d}).\]

Observe that for any fixed $\eta_d$, we have
\begin{align*}
\mathbb{P}(\mathfrak{X}^{\{f'_n\}}(\mathfrak{S})=\mathfrak{T}\mid \eta_d)
&=\mathbb{P}(\mathfrak{X}^{\{f_n\}}(\mathfrak{S}^{\eta_d})=\mathfrak{T})\\
&=\mathbb{P}(\mathfrak{X}_{\sim_d}(\mathfrak{S}^{\eta_d})=\mathfrak{T})\\
&=\mathbb{P}(\mathfrak{X}(\mathfrak{S})=\mathfrak{T}).
\end{align*}

This shows both eq-symmetry and that $\mathfrak{X}^{\{f'_n\}}=_{\mathcal{D}}\mathfrak{X}$.
\end{proof}

\subsection{Lifting Over Equivalence Relations with Infinitely Many Classes}

\begin{theorem}\label{thm:lift_infinite}
  Let $\mathcal{K}$ be a Fra\"iss\'e class in the signature $\mathcal{L}$ with $2$-DAP and with $\odap$-DAP up to a strongly orderly sequence $\sim_1,\ldots,\sim_d$, where $\#(\sim_d)$ is infinite.   Suppose that $\mathfrak{X}$ is a $\mathcal{K}$-exchangeable $\mathcal{L}'$-structure and that there are eq-symmetric functions $\{f_n\}$ so that $\mathfrak{X}^{\{f_n\}}=_{\mathcal{D}}\mathfrak{X}_{\mathrm{dbl}}$.
Then there are eq-symmetric functions $\{f'_n\}$ so that $\mathfrak{X}^{\{f'_n\}}=_{\mathcal{D}}\mathfrak{X}$.
\end{theorem}
\begin{proof}

Recall the model $\mathfrak{S}_{\sim_d}$ defined above.  To a tuple $s\subseteq|\mathfrak{S}|$, we have a corresponding tuple $s_{\sim_d}\subseteq|\mathfrak{S}_{\sim_d}|$, where $s_{\sim_d}=s\cup U_{d,\mathfrak{S}|_s}$.  There is a natural bijection between $\bigcup_{y\in s}E(y)$ and $\bigcup_{y\in s_{\sim_d}}E(y)$: if $y\in s\setminus V^{\mathfrak{S}}$ then we map any $[y]_\sim$ to itself; if $y\in s\cap V^{\mathfrak{S}}$ and $[y]_\sim\in E(y)$ then either $\sim$ is $=$ (and we map $[y]_\sim$ to itself), $\sim$ is some $\sim_i$ so that $\sim_i\cap\sim_d^*$ is orthogonal to $\sim_d^*$ (and we map $[y]_\sim$ to itself), or $\sim$ is some $\sim_i$ so that $\sim_i\cap \sim_d^*$ is $\sim_d^*$ (and we map $[y]_\sim$ to $[c]_\sim$ where $c$ is the eqvalence class of $y$ in $\sim_d$.
Note that this case includes mapping $[y]_{\sim_d}$ to $[c]_=$).  This is clearly surjective as every member of $E(y')$ has one of these forms for some $y'\in s_{\sim_d}$.

Consider an antichain $\tau\in B(s_{\sim_d})$ and note that it need not correspond to an antichain from the collection of blurs $B(s)$.  For instance, when $c$ is the equivalence class of $y$ in $\sim_d$, $\{[y]_=,[c]_=\}\in B(\{y,c\})$ corresponds to $\{[y]_=,[y]_{\sim_d}\}\subseteq E(y)$, which is not an antichain.  We fix this by associating each antichain $\tau\in B(s_{\sim_d})$ with the most refined available option in $B(s)$: given $\tau\in B(s_{\sim_d})$, we define $\hat\tau\in B(s)$ to consist of those $c\in\tau$ such that there is no $c'\in\tau$ with $c'\subseteq c$.

Given $\tau\in B(s)$, we put $T^\tau_{\sim_d}=\{\tau'\in B(s_{\sim_d})\mid \hat\tau'=\tau\}$.  From the random variable $\xi_\tau$ we need to construct $\xi_{\tau'}$ for each $\tau'\in T^\tau_{\sim_d}$ and, for each $\tau'\in T^\tau_{\sim_d}$ other than $\tau$, we need a random ordering $\prec_{\tau'}$.  For each $\tau'\in T^\tau_{\sim_d}\setminus\{\tau\}$, we write $O_{\tau'}$ for the set of functions from random orderings of $\tau$ to random orderings of $\tau'$, and we assign this set the uniform measure.  Fix a measure-preserving function $\pi_\tau:[0,1]^{|T^\tau_{\sim_d}|}\times\prod_{\tau'\in T^\tau_{\sim_d}\setminus\{\tau\}}O_{\tau}\rightarrow[0,1]$.  Fix bijections $\iota_0:T^\tau_{\sim_d}\rightarrow[0,T^r_{\sim_d}-1]$ and $\iota_1:(T^\tau_{\sim_d}\setminus\{\tau\})\rightarrow[T^r_{\sim_d},2T^r_{\sim_d}-1]$.

We now define $f'_{|s|}(\mathfrak{S}|_s,(\xi_\tau)_{\tau\in B(s)},(\prec_\tau)_{\tau\in B(s)},(\eta_r\upharpoonright s))$.  Given $\mathfrak{S}|_s$, we have the embedding $\pi:\mathfrak{S}|_s\rightarrow\mathfrak{S}^-_{\sim_d}$ given by Lemma \ref{thm:inf_embedding}.  For each $\tau'\in B(s_{\sim_d})$ corresponding to $\tau\in B(s)$, let $\xi'_{\tau'}$ be $(\pi^{-1}_{\tau}(\xi_\tau))_{\iota_0(\tau')}$.  Let \[\prec'_{\tau'}=\left\{\begin{array}{cc}\prec_\tau,&\text{ if }\tau'=\tau,\\ (\pi^{-1}_\tau(\xi_\tau))_{\iota_1(\tau')}(\prec_\tau),& \text{otherwise}.\end{array}\right.\]

We have 
\[\tp_{\mathfrak{X}^{\{f_n\}}}(s_{\sim_d})=f_{|s_{\sim_d}|}(\mathfrak{S}_{\sim_d}|_{s_{\sim_d}},(\xi_\tau)_{\tau\in B(s_{\sim_d})},(\prec_\tau)_{\tau\in B(s_{\sim_d})},(\eta_r)).\]
  For each element $y\in s$, assign a pair 
  \[y'=\left\{\begin{array}{cc}
  (y,[y]_{\sim_d}),& \text{if }y\in V^{\mathfrak{S}},\\ (y,y), & \text{if }y\in E^{\mathfrak{S}},\end{array}\right.\]
   and set $(y_1,\ldots,y_k)\in P$ if and only if $(y'_1,\ldots,y'_k)\in P$.

Then
\begin{align*}
  \mathbb{P}(\mathfrak{X}^{\{f'_n\}}(\mathfrak{S})=\mathfrak{T})
&=\mathbb{P}((\mathfrak{X}^{\{f_n\}}(\mathfrak{S}_{\sim_d}))^-=\mathfrak{T})\\
&=\mathbb{P}((\mathfrak{X}_{\mathrm{dbl}}(\mathfrak{S}_{\sim_d}))^-=\mathfrak{T})\\
&=\mathbb{P}(\mathfrak{X}(\mathfrak{S}^-_{\sim_d})=\mathfrak{T})\\
&=\mathbb{P}(\mathfrak{X}(\mathfrak{S})=\mathfrak{T}).
\end{align*}
\end{proof}

\subsection{Main Theorem}

The proof of the main theorem itself is now straightforward.

\begin{theorem}
Let $\mathcal{K}$ be a Fra\"iss\'e class with $2$-DAP and with $\odap$-DAP up to a strongly orderly sequence of definable equivalence relations $\sim_1,\ldots,\sim_d$.  Let $\mathfrak{X}$ be $\mathcal{K}$-exchangeable.
Then there are eq-symmetric functions $\{f_n\}$ such that $\mathfrak{X}^{\{f_n\}}=_{\mathcal{D}}\mathfrak{X}$.
\end{theorem}
\begin{proof}
We obtain $\mathfrak{X}'$ which is $\mathcal{K}_{\sim_d,\ldots,\sim_1}$-exchangeable by repeated applications of Lemmas \ref{thm:lift_fin} and \ref{thm:lift_infin}.  Since $\mathcal{K}_{\sim_d,\ldots,\sim_1}$ has $\odap$-DAP, by Theorem 3.2 from \cite{CraneTowsner2015}, $\mathfrak{X}_{\sim_d,\ldots,\sim_1}=_{\mathcal{D}}\mathfrak{X}^{\{f_n\}}$ for some Borel functions $\{f_n\}$.

We then repeatedly apply Theorems \ref{thm:lift_finite} and \ref{thm:lift_infinite} to obtain eq-symmetric functions $\{f'_n\}$ so that $\mathfrak{X}=_{\mathcal{D}}\mathfrak{X}^{\{f'_n\}}$.
\end{proof}

\bibliography{refs}
\bibliographystyle{abbrv}

\end{document}